\documentclass[12pt,reqno]{amsart}
\advance\hoffset by -0.5 truecm
\usepackage{amssymb}
\newtheorem{Theorem}{Theorem}[section]
\newtheorem{Lemma}[Theorem]{Lemma}
\newtheorem{Sublemma}[Theorem]{Sublemma}
\newtheorem{Corollary}[Theorem]{Corollary}

\newtheorem{Proposition}[Theorem]{Proposition}

\newtheorem{Definition}[Theorem]{Definition}

\newtheorem{Remark}[Theorem]{Remark}

\def \dim{{\mbox {dim}}\,}

\begin{document}
\title[Fiberwise symmetrizations]{Fiberwise symmetrizations for variational problems on fibred manifolds}

\author[C. Sung]{Chanyoung Sung}
 \address{Dept. of Mathematics Education  \\
          Korea National University of Education \\
          Cheongju, Korea
          }
\email{cysung@kias.re.kr}



\date{}


\begin{abstract}
We establish a framework for fiberwise symmetrization to find a lower bound of a Dirichlet-type energy functional in a variational problem on a fibred Riemannian manifold, and use it to prove a comparison theorem of the first eigenvalue of the Laplacian on a warped product manifold.
\end{abstract}

\maketitle


\section{Introduction}

On a Riemannian manifold one often deals with a Dirichlet-type energy functional in a variational problem and needs to find the minimum of such a functional.
Under some geometric conditions its lower bound can be given by the minimum value of the corresponding functional on a isoperimetric region of a space form, where a minimizer turns out to be radially symmetric. Here we find the basic underlying principle that symmetrization reduces energy.

For a Morse function $f$ on a smooth closed Riemannian $m$-manifold $(M^m,g)$ of volume $V$, the spherical rearrangement of $f$ is the radially-symmetric continuous function $f_*:S^m_V\rightarrow \Bbb R$  such that $\{q\in S^m_V| f_*(q)<t \}$ for any $t\in (\min(f), \max(f)]$ is the geodesic ball centered at the south pole $q_0$ with the same volume as $\{x\in M| f(x)<t \}$ where $S^m_V:=(S^m,\textsl{g}_V)$ denotes an $m$-sphere with a metric of constant curvature and volume $V$. By the radial symmetry of a function on $S^m$ we mean its invariance under the standard action of $O(m+1)$ around the axis passing through the two poles.
Throughout the paper, we shall adopt these notations along with $(S^m, g_{_{\Bbb S}})$ which denotes the standard round $m$-sphere with constant curvature 1 and volume $V_m$.

The reason we require $f$ to be a Morse function is that the volume of $\{x\in M| f(x)<t \}$ as a function of $t$ is a continuous function which is smooth except a finite number of points so that the resulting $f_*$ is a well-defined piecewise-smooth continuous function.
The good thing about $f_*$ is that it is not only a radially-symmetric continuous function but also its Dirichlet energy $||\nabla f_*||_{L^2}$ can be estimated by $||\nabla f||_{L^2}$ when $(M^m,g)$ is favourable, while its $L^2$-norm $||f_*||_{L^2}$ is always the same as $||f||_{L^2}$. We aim to show that this method of symmetrization can be done fiberwisely in case of a fibred manifold, and can be applied to estimate lower bounds of Dirichlet-type energy functionals on a fibred manifold.

The first obstacle for this task is that it is generally not possible to approximate $C^1$-closely any smooth function by a fiberwise Morse function, but it turns out that it can be approximated by a so-called generic-fiberwise Morse function. As will be explained in Section \ref{elle}, $F$ being a generic-fiberwise Morse function on $N\times M$ with fiber $M$ means that there exists an open dense subset $N_0$ of $N$ with measure-zero complement such that $F$ restricted to each fiber over $N_0$ is a Morse function. We shall prove that the subset of such functions are dense in $C^\infty(N\times M)$ with $C^2$-topology and our main results are summarized as :
\begin{Theorem}\label{forgive-me}
Let $(M^m,g)$ and $(N^n,h)$ be smooth closed Riemannian manifolds such that $M$ has volume $V$. If $F:N\times M\stackrel{C^\infty}{\rightarrow} \Bbb R$ is a generic-fiberwise Morse function, then there exists a Lipschitz continuous function $F_{\bar{*}}\in L_1^p(N\times S^m_V)$ for any $p\geq 1$ which is smooth on an open dense subset with measure-zero complement such that for $s\in N_0$ $$F_{\bar{*}}|_{\{s\}\times S^m_V}=(F|_{\{s\}\times M})_*.$$
Moreover
$$||d^SF_{\bar{*}}||_{L^p} \leq C ||d^M F||_{L^p}\ \ \ \textrm{and}\ \ \ ||d^NF_{\bar{*}}||_{L^p} \leq ||d^N F||_{L^p}$$ while  $||F_{\bar{*}}||_{L^p}=||F||_{L^p}$, where $C$ is a constant depending on $(M,g)$ and is equal to $(\frac{V_m}{V})^{\frac{1}{m}}\sqrt{\frac{m-1}{k}}$ if Ricci curvature $Ric_g\geq k>0$.
\end{Theorem}
Here obvious product metrics are assumed on $N\times M$ and $N\times S^m_V$ and we split the exterior derivative $d$ on $N\times M$ and $N\times S^m$ as $d^N+d^M$ and $d^N+d^S$ respectively.
Namely $d^N=p_1\circ d$ and  $d^M=p_2\circ d$ where $p_1:T^*(N\times M)\rightarrow \pi_1^*(T^*N)$ and  $p_2:T^*(N\times M)\rightarrow \pi_2^*(T^*M)$ are projections according to the obvious decomposition $T^*(N\times M)\simeq \pi_1^*(T^*N)\oplus \pi_2^*(T^*M)$, and similarly for $d=d^N+d^S$ on $N\times S^m$.

We shall also consider fiberwise Euclidean or hyperbolic symmetrization in Section \ref{SWYang}, and indeed similar results are obtained in the same way, assuming that $M$ is flat or hyperbolic respectively.

One can apply these results to some variational problems on fibred manifolds, such as the first eigenvalue problem and the Yamabe problem, and estimate the lower bounds of corresponding functionals on fibered manifolds.
As an example of a test case, we shall estimate a lower bound of the first eigenvalue $\lambda_1$ of the Laplacian on a warped product manifold.
\begin{Theorem}\label{Yoon-1}
Let $(M^m,g)$ and $(N^n,h)$ be smooth closed Riemannian manifolds for $m\geq 2$ such that $M$ has Ricci curvature $\textrm{Ric}_g \geq m-1$. Then for any smooth positive function $\rho$ on $N$,
$$\lambda_1(N\times M,h+\rho^2g) \geq \lambda_1(N\times S^m,h+\rho^2g_{_{\Bbb S}}).$$
\end{Theorem}
The case of nontrivial bundles are more involved, so we shall consider a generalization of the above theorem to some nontrivial bundles along with more interesting examples in a forthcoming paper \cite{Sung1}. We shall also investigate applications to the estimation of Yamabe constants elsewhere \cite{Sung2}.

\section{Preliminaries}

\subsection{Conventions and Notations}

Here are collected some common conventions and notations to be used throughout this paper.

First of all, every manifold is assumed to be smooth and connected unless otherwise stated.
When $X$ is a $k$-dimensional Riemannian manifold (possibly with boundary), $\mu(X)$ denotes the $k$-dimensional volume of $X$, and $\chi_{_S}$ denotes the characteristic function of a subset $S$.

In case that a Lie group $G$ has a smooth left (or right) action on a smooth manifold $X$, it induces a left (or right) action on $C^\infty(X)$ respectively defined by
\begin{eqnarray}\label{BHCP}
(\mathfrak{g}\cdot F)(x):=F(\mathfrak{g}^{-1}\cdot x)\ \ \ \ \textrm{or}\ \ \ \ (F\cdot \mathfrak{g})(x):=F(x\cdot \mathfrak{g}^{-1})
\end{eqnarray}
for $F\in C^\infty(X)$, $\mathfrak{g}\in G$ and $x\in X$.

Given a function $F$ on a product manifold $N\times X$, we shall regard $F$ as a family of functions on $X$ parametrized by $N$, in which case we denote $F|_{\{s\}\times X}$ for $s\in N$ by $F_s$ for notational simplicity. The reader is requested to be cautious that it should not be misunderstood as the abbreviated notation of the partial derivative $\frac{\partial F}{\partial s}$.

Given a map $f$, its inverse image of a set $S$ is denoted by $f^{-1}S$, and
the open ball of radius $\epsilon>0$ and center $p$ in any metric space is denoted by $B_\epsilon(p)$.

\subsection{Weak derivative and Sobolev space}
On a smooth Riemannian manifold $(M,g)$, $L_1^p(M)$ for $p\geq 1$ is the completion of $$\{f\in C^\infty(M)| \ ||f||_{L_1^p}:=(\int_M (|f|^p+|df|_g^p)\ d\mu_g)^{\frac{1}{p}}<\infty\}$$ with respect to $L_1^p$-norm $||\cdot||_{L_1^p}$.
Sometimes it's not easy to determine whether a non-smooth function belongs to it or not. However, on a smooth closed $M$, $L_1^p(M)$ is not only independent of choice of metric $g$ but also characterized as the space of weakly differentiable functions on $M$ with bounded $L_1^p$-norm.

A locally integrable function $f:M\rightarrow \Bbb R$, i.e. in $L^1_{loc}(M)$, is called weakly differentiable if it has all of its weak partial derivatives of 1st order in a local coordinate. Here a weak derivative is meant in distributional sense, and required to be in $L^1_{loc}$. This condition is independent of the choice of a local coordinate,\footnote{Strictly speaking, it is enough for the coordinate change to be a $C^1$-diffeomorphism.} and a weak derivative, if it exists, is unique up to a set of measure zero.

In a special case of an open interval in $\Bbb R$, weak differentiability is equivalent to absolute continuity, where a weak derivative is just the ordinary (strong) derivative defined almost everywhere. A function $f:I\rightarrow \Bbb R$ on an interval $I$ is absolutely continuous if for any $\epsilon>0$, there exists $\delta>0$ such that for any finite number of subintervals $I_j=[a_j,b_j]\subset I$ satisfying $\sum_j|a_j-b_j|<\delta$ and $(a_j,b_j)\cap (a_k,b_k)=\emptyset$ for $j\ne k$, $\sum_j|f(a_j)-f(b_j)|<\epsilon.$  In case that $I$ is a bounded closed interval, the absolute continuity on $I$ is equivalent to that $f$ is weakly differentiable on the interior of $I$ and its weak derivative belongs to $L^1(I)$.
In higher dimensions, we have the following characterization for weak differentiability :
\begin{Proposition}\label{Sewoong}
Let $\Omega\subset \Bbb R^n$ be a domain. Then $f\in L^1_{loc}(\Omega)$ is weakly differentiable iff it is equivalent to a function that is absolutely continuous on almost every straight line
parallel to the coordinate axes and whose 1st order partial derivatives (which consequently exist a.e. and are measurable) are locally integrable in $\Omega$.
\end{Proposition}
\begin{proof}
There are many books containing the proof of this. The reader may refer to \cite{Evans, Mazja, morrey, ziemer}.
\end{proof}

As a result, the weak 1st order partial derivatives of a weakly differentiable function is equivalent to its ordinary partial derivatives defined almost everywhere.

\subsection{Isoperimetric inequality}

The symmetrization methods are closely related to the isoperimetric problem. One can apply the results on the isoperimetric comparison to estimating Dirichlet-type energy functionals. The followings are the facts we shall use later.
\begin{Definition}
On a smooth closed Riemannian $m$-manifold $(M,g)$ with volume $V$ for $m\geq 2$, and for any $\beta \in (0,1)$, let $$W_{\beta}
=\{ \Omega \subset M| \Omega \textrm{ is open with smooth } \partial\Omega \textrm{ and }\mu(\Omega )=\beta V \}.$$ The isoperimetric function $h_{(M,g)}:(0,1)\rightarrow \Bbb R$ of $(M,g)$
is defined by
$$h_{(M,g)}(\beta ):=
\inf \{ \frac{\mu (\partial \Omega )}{V} | \Omega \in
W_{\beta} \} .$$
Similarly the isoperimetric profile $I_{(M,g)}:(0,V)\rightarrow \Bbb R$ of  $(M,g)$ is defined as
$$I_{(M,g)}(v):=
\inf \{ \mu(\partial \Omega) | \Omega \in W_{\frac{v}{V}} \}=Vh_{(M,g)}(\frac{v}{V}).$$
\end{Definition}
Obviously $$I_{(M,g)}(v)=I_{(M,g)}(V-v)\ \ \ \textrm{and}\ \ \ h_{(M,a^2 g)} = \frac{1}{a} h_{(M,g)}$$ for a constant $a>0$.
It is well-known that the isoperimetric profile is continuous and satisfies $$\lim_{v\rightarrow 0}\frac{I_{(M,g)}(v)}{v^{\frac{m-1}{m}}}=\gamma_m$$
where $\gamma_m=\frac{\mu(S^{m-1}, g_{_{\Bbb S}})}{\mu(B_1(0))^{\frac{m-1}{m}}}$ is the classical $m$-dimensional isoperimetric constant in $(\Bbb R^m,g_{_{\Bbb E}})$.(\cite{MJ})
Thus there exists a constant $\gamma>0$ so that for any sufficiently small $v>0$
$$\gamma^{-1} v^{\frac{m-1}{m}} \leq   I_{(M,g)}(v) \leq \gamma v^{\frac{m-1}{m}},$$ and so if $(M',g')$ is another smooth closed Riemannian $m$-manifold with the same volume $V$ as $M$, then there exists a constant $\gamma'>0$ so that
\begin{eqnarray}\label{ibs-cnu}
(\gamma')^{-1} I_{(M,g)}(v) \leq   I_{(M',g')}(v) \leq \gamma' I_{(M,g)}(v)
\end{eqnarray}
for any $v\in (0,V)$.

The L\'{e}vy-Gromov comparison theorem states that if $(M,g)$ satisfies $Ric_g \geq m-1$,
then $$h_{(M,g)}(\beta ) \geq h_{(S^m, g_{_{\Bbb S}})} (\beta )$$ for any $\beta$. (In fact a little stronger inequality holds by B\'{e}rard-Besson-Gallot \cite{Gallot}.)
Simply-connected complete Riemannian manifolds with constant curvature which are $(S^m,g_{_{\Bbb S}})$, $(\Bbb R^m,g_{_{\Bbb E}})$, $(\Bbb H^m,g_{_{\Bbb H}})$ are called isoperimetric model spaces, because the isoperimetric inequality on them is saturated by the most symmetric domains. Namely any geodesic ball in it has the least ``boundary area"  among all the domains with the same volume.

\section{Fiberwise Morse function}\label{elle}

To develop our main tool, the use of Morse theory is crucial.
On a smooth manifold $M$, a smooth function $f:M\rightarrow \Bbb R$ is defined to be a Morse function if it
has only non-degenerate critical points, meaning that $df$ which is a smooth section of $T^*M$ is transverse to the zero section. If $M$ is compact, then any Morse function on it must have finite critical points, and moreover the set of Morse functions on $M$ forms an open dense subset of $C^2(M,\Bbb R)$.

Recall that the $C^k$-topology for $k\geq 0$ on the vector space of smooth real-valued functions on a compact $n$-manifold $X$ (possibly with boundary) is defined as follows.
First we choose any smooth Riemannian metric on $X$, and define the $C^k$-norm of $f$ as
$$||f||_{C^k}:=\sum_{i=0}^k\sup_X|\stackrel{i}{\overbrace{\nabla\cdots\nabla}} f|$$ where the covariant derivative $\nabla$ and $|\cdot|$ are computed using the chosen metric. The topology induced by this norm is the $C^k$-topology which is in fact independent of the particular choice of a metric.

The notion of ``almost everywhere" means everywhere except a subset of measure zero. In case of a smooth manifold, we shall use the measure induced by a Riemannian metric on it, and obviously the measure zero property does not depend on the particular choice of a Riemannian metric. Moreover the measure zero property can be defined without recourse to a Riemannian metric. One takes a coordinate covering and compute its measure in Lebesgue measure of each local coordinate.
One can easily see that being of measure-zero w.r.t. Lebesgue measure in each local coordinate is equivalent to having zero Riemannian volume.

The following is a key proposition to enter our main argument. We shall consider any element of $C^\infty(N\times M, \Bbb R)$ as a smooth family of smooth real-valued functions on $M$ parametrized by $N$, and recall that $F|_{\{s\}\times M}$ for $F:N\times M\rightarrow \Bbb R$ is shortly denoted by $F_s$.
\begin{Theorem}\label{morse}
Let $M$ and $N$ be smooth closed manifolds and $C^\infty(N\times M, \Bbb R)$ be endowed with the $C^2$-topology. If $F\in C^\infty(N\times M, \Bbb R)$, then for any open neighborhood $U\subseteq C^\infty(N\times M, \Bbb R)$ of $F$ there exist $\tilde{F}\in U$ and an open dense subset  $N_0\subseteq N$ such that the complement $N_0^c$ of $N_0$ is of measure zero and $\tilde{F}_s$ for any $s\in N_0$ is a Morse function on $M$.
\end{Theorem}
\begin{proof}
Take any smooth embedding $\imath:M\hookrightarrow \Bbb R^k$ for some $k$ and an open ball $B_\epsilon(0)$ in $\Bbb R^k$, and let $x_1,\cdots,x_k$ be the usual coordinate functions on $\Bbb R^k$. By the well-known fact in Morse theory \cite{pollack}, $$E_s:=\{(a_1,\cdots,a_k)\in B_\epsilon(0)|F_s+\imath^*(\sum_{j=1}^ka_jx_j)  \textrm{ is a Morse function on }M\}$$
is an open dense subset of $B_\epsilon(0)$, whose complement in $B_\epsilon(0)$ is of measure-zero. Define $$E:=\cup_{s\in N}(\{s\}\times E_s)\subseteq N\times B_\epsilon(0).$$

Because the set of Morse functions on $M$ is open in $C^\infty(M)$ with respect to $C^2$-topology, and $F_s$ is smoothly-varying with respect to $s\in N$, if $F_s+\imath^*(\sum_ja_jx_j)$ is a Morse function on M, then there exists an open neighborhood $U(s)\subset N$ of $s$ and an open ball $B(a)\subset B_\epsilon(0)$ around $a:=(a_1,\cdots,a_k)$ such that $F_t+\imath^*(\sum_jb_jx_j)$ for any $t\in U(s)$ and any $(b_1,\cdots,b_k)\in B(a)$ is also a Morse function on $M$.
So any $(s,a)\in E$ has an open neighborhood contained in $E$, implying that $E$ is an open subset of $N\times B_\epsilon(0)$.

Moreover a (relatively) closed subset $E^c\subset N\times B_\epsilon(0)$ has measure zero with respect to the product measure\footnote{In fact, as a Borel measure, the product measure $\mu_h\times \mu_{g_{_{\Bbb E}}}$ is equal to the Riemannian measure $\mu_{h+g_{_{\Bbb E}}}$ by Fubini's theorem in $\Bbb R^{k+\dim N}$.} $\mu_h\times \mu_{g_{_{\Bbb E}}}$ where $h$ is a Riemannian metric on $N$ and $g_{_{\Bbb E}}$ is the Euclidean metric on $B_\epsilon(0)$, since
\begin{eqnarray*}
\int_{N\times B_\epsilon(0)}\chi_{_{E^c}}\ d(\mu_h\times\mu_{g_{_{\Bbb E}}}) &=&  \int_{N}\int_{B_\epsilon(0)}\chi_{_{E^c}}\ d\mu_{g_{_{\Bbb E}}} d\mu_h\\ &=&  \int_{N}0\ d\mu_h\\ &=& 0.
\end{eqnarray*}
We claim that for almost every $a\in B_\epsilon(0)$, $H_a:=E^c\cap (N\times \{a\})$ considered as a (closed) subset of $N$ has measure 0. Using the product measure $\mu_h\times \mu_{g_{_{\Bbb E}}}$ again,
\begin{eqnarray*}
0 &=& \int_{N\times B_\epsilon(0)} \chi_{_{E^c}} \ d(\mu_h\times\mu_{g_{_{\Bbb E}}}) \\ &=& \int_{B_\epsilon(0)}\left(\int_N \chi_{_{E^c}} \ d\mu_h\right) d\mu_{g_{_{\Bbb E}}}\\ &=& \int_{B_\epsilon(0)}\mu_h(H_a)\ d\mu_{g_{_{\Bbb E}}}(a).
\end{eqnarray*}
Since the measure $\mu_h(H_a)$ of $H_a$ is nonnegative, it must be zero for almost every $a=(a_1,\cdots,a_k)$. We choose such $a$ that $\mu_h(H_a)=0$ and  define $$\tilde{F}:=F+\imath^*(\sum_ja_jx_j)\ \ \ \ \textrm{and}\ \ \ \ N_0:=N-H_a.$$
Thus we have proved that $\tilde{F}_s$ is a Morse function on $M$ for $s\in N_0$, and $N_0$ is open and dense in $N$ since its complement is a closed subset of measure zero. That $\tilde{F}$ belongs to $U$ is achieved by taking any sufficiently small $\epsilon>0$.



We will find an open neighborhood $B_\epsilon(F)$ contained in $\mathcal{E}$. Let's define any $C^2$-norms $||\cdot||_{C^2}$ for $C^\infty(N, \Bbb R)$ and $C^\infty(M, \Bbb R)$ respectively. Those norms naturally give a $C^2$-norm $||\cdot||_{C^2}$ for $C^\infty(N\times M, \Bbb R)$ such that embeddings $\imath:C^\infty(M, \Bbb R)\hookrightarrow C^\infty(N\times M, \Bbb R)$ and $\jmath:C^\infty(N, \Bbb R)\hookrightarrow C^\infty(N\times M, \Bbb R)$ are isometries with respect to these norms. Repeating the above argument with $f_i=0\ \forall i$, one can get a $N(F)$ satisfying  $N(F)\subseteq \mathcal{E}$. Define $$\mathcal{U}_\epsilon:=\{f\in C^\infty(M, \Bbb R)|\   ||f||_{C^2}<\epsilon\}$$ and $$F+\mathcal{U}_\epsilon:=\{F+f|f\in U_\epsilon\}\subset C^\infty(N\times M, \Bbb R).$$ For any sufficiently small $\epsilon>0$, $F+\mathcal{U}_\epsilon\subseteq N(F)$ and hence $F+\mathcal{U}_\epsilon\subseteq \mathcal{E}$. If $\tilde{F}\in C^\infty(N\times M, \Bbb R)$ satisfies $||\tilde{F}-F||_{C^2}<\epsilon$, then for any $s\in N$ $$\tilde{F}|_{\{s\}\times M}-F|_{\{s\}\times M}\in \mathcal{U}_\epsilon$$ so that it is a Morse function on $M$. Thus taking $$B_\epsilon(F):=\{\tilde{F}\in C^\infty(N\times M, \Bbb R)|\ ||\tilde{F}-F||_{C^2}<\epsilon\},$$ we have $B_\epsilon(F)\subset \mathcal{E}$ for any sufficiently small $\epsilon>0$.
\end{proof}
We shall call a smooth function $F:N\times M\rightarrow \Bbb R$ \textit{generic-fiberwise Morse function} if there exists an open dense subset $N_0$ of $N$ such that $N_0^c$ is of measure zero and $F_s$ for all $s\in N_0$ is a Morse function on $M$. One may take such $N_0$ to be the largest one, i.e. the union of all such $N_0$ and we denote it by $N_0(F)$. If $N_0(F)$ is equal to $N$, we shall call $F$ \textit{fiberwise Morse function}.

In case that a Lie group $G$ has a smooth left (or right) action on both $N$ and $M$, it acts on $N\times M$ by the obvious extension, and induces a left (or right) action on $C^\infty(N\times M)$ as in (\ref{BHCP}).
\begin{Corollary}\label{morse-1}
Under the same hypothesis as Theorem \ref{morse}, further suppose that a compact Lie group $G$ acts on $N$ and $M$ smoothly, and $F:N\times M\rightarrow \Bbb R$ is invariant under this action. Then one can choose $\tilde{F}$ in Theorem \ref{morse} such that the $G$-orbit of $\tilde{F}$ also remains in $U$. If the action on $M$ is trivial, then $\tilde{F}$ can be chosen to be $G$-invariant.
\end{Corollary}
\begin{proof}
First, the 2nd statement is obvious, because the added function $\imath^*(\sum_ia_ix_i)$ in the above proof is invariant under the $G$-action.

To prove the 1st statement, let's just use the left action notation for a function.
From the compactness of $G$ and $M$
it follows that
\begin{eqnarray*}
\sup_{\mathfrak{g}\in G}||F-\mathfrak{g}\cdot \tilde{F}||_{C^2} &=& \sup_{\mathfrak{g}\in G}||\mathfrak{g}\cdot(\imath^*(\sum_ja_jx_j))||_{C^2}\\
&=& \sup_{\mathfrak{g}\in G}||\sum_ja_j\ \mathfrak{g}\cdot(\imath^*x_j)||_{C^2}
\end{eqnarray*}
can be made arbitrarily small by taking $(a_1,\cdots,a_k)$ sufficiently small.
\end{proof}



\section{Spherical rearrangement}
Recall the definition of $S^m_V$ with $\textsl{g}_V={\left( \frac{V}{V_m} \right) }^{\frac{2}{m}}g_{_{\Bbb S}}$, and we let $r(q)\in [0, R_V]$ for $q\in S^m_V$ denote the distance from the south pole $q_0$ to $q$, where $R_V$ denotes the diameter of $S^m_V$.
In this section, $(M,g)$ is a smooth closed Riemannian $m$-manifold of volume $V$.

\subsection{Unparametrized spherical rearrangement}
For a Morse function $f$ on $(M,g)$, the spherical rearrangement of $f$ is the radially-symmetric continuous function $f_*:S^m_V\rightarrow \Bbb R$  such that
$\{q\in S^m_V| f_*(q)<t \}$ for any $t\in (\min(f), \max(f)]$ is the geodesic ball centered at $q_0$ with the same volume as $\{x\in M| f(x)<t \}$.  Namely $f_*$ is defined by that $f_* (q)=t$ for $q\in S^m_{V}$ if and only if the volume of the geodesic ball of radius $r(q)$ in $S^m_{V}$ equals $\mu(\{x\in M| f(x)<t \})$. By abuse of notation we regard $f_*$ also as a function on $[0, R_V]$ (or $\Bbb R$ by an obvious continuous extension).
Thus $f_*:[0,R_V]\rightarrow \Bbb R$ is strictly increasing.

If $t_0$ is a regular value of $f$ then the function $t\rightarrow \mu(\{p\in M| f(p)<t \})$ is smooth at $t_0$, and $f_*$ is smoothly differentiable around $f_*^{-1}(t_0)$ with $t_0$ being a regular value of $f_*$. We shall prove this in detail in the next subsection.
In this case note  that $|\nabla f_* |$ is constant along $f_*^{-1} (t_0)$ since $f_*$ is radially symmetric.

Note two important facts we shall often use. First, for any $k> 0$ and $(a,b)\subset \Bbb R$
\begin{eqnarray}\label{gorba}
\int_{\{x\in M| a<f(x)<b \}}f^k\ d\mu_g=\int_{\{q\in S^m_{V}|a<f_*(q)<b\}} (f_*)^k\ d\mu_{\textsl{g}_V}
\end{eqnarray}
where $(\cdot)^k$ means $|\cdot|^k$ if $k$ is not an integer. Secondly, the coarea formula :
$$ \int_{\{x\in M| a<f(x)<b \}} \psi\ d\mu_g = \sum_{i=0}^{l}\int_{t_{i}}^{t_{i+1}} \left( \int_{f^{-1} (t) }
\psi| \nabla f|^{-1} d\sigma_t \right) dt$$ for any $\psi\in L^1(M)$ where $d\sigma_t$ denotes the volume element of $f^{-1} (t)$ with the induced Riemannian metric,  $t_1<\cdots<t_l$ are all the critical values of $f$ lying in $(a,b)$, and we set $t_0:=a, t_{l+1}:=b$.
If no confusion arises, $d\sigma_t$ will always denote the induced volume element of the level-$t$ hypersurface of a function concerned.

Thus if $a$ is a regular value of $f$, then
\begin{eqnarray}\label{WTH}
\int_{f^{-1} (a)} {| \nabla f |}^{-1} d\sigma_{a} &=&\frac{d}{dt}|_{t=a}\mu (\{x\in M| f(x)<t \})\nonumber \\
&=&\frac{d}{dt}|_{t=a} \mu (\{q\in S^m_V| f_*(q) <t \})\nonumber \\
&=& \int_{f_*^{-1} (a)} {| \nabla f_* |}^{-1} d\sigma_{a}.
\end{eqnarray}

The followings are essential ingredients for later arguments.
\begin{Proposition}\label{Firstlady-1}
Let $$\mathcal{F}:=\{f:M\stackrel{C^\infty}{\rightarrow} \Bbb R|f \textrm{ is a Morse function.}\}$$ and
$$\tilde{\mathcal{F}}:=\{f_{*}: S^m_V\rightarrow \Bbb R|f\in \mathcal{F}\}$$ be endowed with $C^0$-topology.
Then the spherically-rearranging map $\Phi:\mathcal{F}\rightarrow \tilde{\mathcal{F}}$ given by $\Phi(f)=f_*$ is Lipschitz continuous with Lipschitz constant 1,
i.e. for any $f, g\in \mathcal{F}$ $$||g_{*}-f_{*}||_\infty\leq ||g-f||_\infty.$$
\end{Proposition}
\begin{proof}
Suppose not.
Then  there exist $f$ and $g$ in $\mathcal{F}$ such that $$||g_{*}-f_{*}||_\infty>||g-f||_\infty.$$
So there exists $r_0\in (0,R_V)$ such that $$|g_{*}(r_0)-f_{*}(r_0)|>||g-f||_\infty.$$
($g_{*}$ and $f_{*}$ are understood as functions on $[0,R_V]$.)
Without loss of generality we may assume that $$f_{*}(r_0)> g_{*}(r_0)+||g-f||_\infty$$ holds, and we set $a:=g_{*}(r_0)$.
Then
\begin{eqnarray*}
\mu(\{y\in S^m_V|r(y)<r_0\}) &=& \mu(\{y\in S^m_V|g_{*}(r(y))< a  \}) \\ &=& \mu(\{x\in M|g(x)<  a  \})\\ &\leq& \mu(\{x\in M|f(x)<  a+||g-f||_\infty  \}) \\
&=& \mu(\{y\in S^m_V|f_{*}(r(y))<  a+||g-f||_\infty  \})\\ &<& \mu(\{y\in S^m_V|f_{*}(r(y))<  f_{*}(r_0) \})\\ &=& \mu(\{y\in S^m_V|r(y)<r_0\})
\end{eqnarray*}
which yields a contradiction.
\end{proof}

\begin{Proposition}\label{sejong-brt}
There exists a constant $C_1>0$ depending only on $(M,g)$ so that for any Morse function $f:M\rightarrow \Bbb R$, any interval $(a,b)$, and any $p\geq 1$
$$\int_{f_*^{-1}(a,b)} {| \nabla f_*| }^p\ d\mu_{\textsl{g}_V} \leq (C_1)^p \int_{f^{-1}(a,b)} {| \nabla f | }^p\ d\mu_g ,$$ and for any regular value $a_0=f_*(r_0)$ of $f$  $$|\nabla f_*(r_0)|\leq C_1\max_{f^{-1}(a_0)}|\nabla f|.$$
\end{Proposition}
\begin{proof}
To prove the 1st inequality we use the idea in \cite{Petean} where the case of $p=2$ is proved.
Let $t$ be a regular value of $f$. Applying H\"{o}lder's inequality to $$1={| \nabla f|}^{-\frac{1}{q}}{| \nabla f |}^{\frac{1}{q}}={| \nabla f|}^{-\frac{1}{q}}{| \nabla f |}^{\frac{p-1}{p}}$$ on $f^{-1} (t)$ where $\frac{1}{p}+\frac{1}{q}=1$, we have
\begin{eqnarray*}
\int_{f^{-1} (t)}| \nabla f|^{p-1} \  d\sigma_t  &\geq& {\left( \mu (f^{-1} (t)) \right)}^p
{\left( \int_{f^{-1} (t)} {| \nabla f |}^{-1} d\sigma_t \right)}^{-\frac{p}{q}}\nonumber \\
&=& {\left( \mu (f^{-1} (t)) \right)}^p {\left( \int_{f^{-1} (t)} {| \nabla f |}^{-1} d\sigma_t \right)}^{-p+1},
\end{eqnarray*}
where $\mu(\cdot)$ denotes $(m-1)$-dimensional volume.


Using the facts that $\{x\in M| f(x)< t \}$ is an open submanifold of $M$ with smooth boundary $f^{-1} (t)$, and the values of the isoperimetric profile on the round sphere are realized by geodesic balls, and (\ref{ibs-cnu}), we have
\begin{eqnarray}\label{book-repay}
\mu (f^{-1} (t)) &\geq&  I_{(M,g)}(\mu(\{x\in M| f(x)<t \})) \nonumber\\ &\geq& \frac{1}{C_1} I_{S^m_V}(\mu(\{x\in M| f(x)<t \})) \nonumber\\ &=& \frac{1}{C_1}\mu( f_*^{-1}(t))
\end{eqnarray}
for a constant $C_1>0$.

Now we combine the above two inequalities with the coarea formula and (\ref{WTH}) to get
\begin{eqnarray*}
\int_{f^{-1}(a,b)}  {| \nabla f | }^p\ d\mu_g &=& \sum_{i=0}^{l}\int_{t_{i}}^{t_{i+1}} \left( \int_{f^{-1} (t) }
| \nabla f|^{p-1} \ d\sigma_t \right) dt\\ &\geq& \sum_{i=0}^{l}\int_{t_{i}}^{t_{i+1}}
{\left( \mu (f^{-1} (t)) \right)}^p
{\left( \int_{f^{-1} (t)} {| \nabla f|}^{-1} d\sigma_t \right)}^{-p+1}dt\\
&\geq& {\left( \frac{1}{C_1} \right) }^{p}
\sum_{i=0}^{l}\int_{t_{i}}^{t_{i+1}}
{\left( \mu (f_*^{-1} (t)) \right)}^p
{\left( \int_{f_*^{-1} (t)} {| \nabla f_* |}^{-1} d\sigma_t \right)}^{-p+1}dt\\
&=&  {\left( \frac{1}{C_1} \right) }^{p}
 \sum_{i=0}^{l}\int_{t_{i}}^{t_{i+1}}  \left( \int_{f_*^{-1} (t) }
| \nabla f_* |^{p-1} \  d\sigma_t  \right)  dt\\
&=& {\left( \frac{1}{C_1} \right) }^{p}
\int_{f_*^{-1}(a,b)}  {| \nabla f_* | }^p\ d\mu_{\textsl{g}_V},
\end{eqnarray*}
where $t_i$ for $i=1,\cdots,l$ are critical values of $f$ on $(t_0,t_{l+1}):=(a,b)$, and the equality of the 4th line is due to the radial symmetry of $f_*$.

Although the 2nd inequality can be derived from the 1st inequality, we can directly get
\begin{eqnarray*}
|\nabla f_*(r_0)| &=& \frac{\mu(f_*^{-1}(a_0))}{\int_{f_*^{-1}(a_0)}|\nabla f_*|^{-1}d\sigma_{a_0}}\\
&=& \frac{\mu(f_*^{-1}(a_0))}{\int_{f^{-1}(a_0)}|\nabla f|^{-1}d\sigma_{a_0}}\\
&\leq&    \frac{\mu(f_*^{-1}(a_0))\int_{f^{-1}(a_0)}|\nabla f|\ d\sigma_{a_0}}{(\mu(f^{-1}(a_0)))^2}\\ &\leq&
C_1\frac{\int_{f^{-1}(a_0)}|\nabla f|\ d\sigma_{a_0}}{\mu(f^{-1}(a_0))}
\\ &\leq& C_1\max_{f^{-1}(a_0)}|\nabla f|.
\end{eqnarray*}

\end{proof}

When $M$ has positive Ricci curvature, we can explicitly compute the above constant $C_1$.
\begin{Proposition}\label{pet}
Suppose that $Ric_g\geq m-1$ for $m\geq 2$ and $f:M\rightarrow \Bbb R$ is a Morse function.
Then for any interval $(a,b)$ and any $p\geq 1$
$$\int_{f_*^{-1}(a,b)} {| \nabla f_*| }^p\ d\mu_{\textsl{g}_V} \leq {\left( \frac{V_m}{V} \right) }^{\frac{p}{m}} \int_{f^{-1}(a,b)} {| \nabla f | }^p\ d\mu_g,$$ and for any regular value $a_0=f_*(r_0)$ of $f$
$$|\nabla f_*(r_0)|\leq  {\left( \frac{V_m}{V} \right) }^{\frac{1}{m}}\max_{f^{-1}(a_0)}|\nabla f|.$$
\end{Proposition}
\begin{proof}
The proof is the same as the above proposition except for (\ref{book-repay}).
By using $Ric_g\geq m-1$ and the fact that the values of the isoperimetric function on the round sphere are realized by geodesic balls, we now have for any regular value $t$ of $f$
\begin{eqnarray}\label{uncultured}
\mu (f^{-1} (t)) &\geq&  V \ h_{(M,g)} \left( \frac{\mu(\{x\in M| f(x)<t \})}{V} \right)\nonumber\\ &\geq&  V \ h_{(S^m,g_{_{\Bbb S}})} \left( \frac{\mu(\{x\in M| f(x)<t \})}{V} \right)\nonumber\\ &=& V{\left( \frac{V}{V_m} \right) }^{\frac{1}{m}}h_{S^m_{V}} \left( \frac{\mu(\{y\in S^m_V| f_*(y)<t \})}{V} \right)\nonumber\\ &=& {\left( \frac{V}{V_m} \right) }^{\frac{1}{m}} \mu( f_*^{-1}(t)).
\end{eqnarray}


\end{proof}

\begin{Corollary}\label{itismyfault}
Suppose that $Ric_g\geq k$ for a constant $k>0$ and $f:M\rightarrow \Bbb R$ is a Morse function.
Then for any interval $[a,b]$ and any $p\geq 1$
$$\int_{f_*^{-1}(a,b)} {| \nabla f_*| }^p\ d\mu_{\textsl{g}_V} \leq {\left( \frac{V_m}{V'} \right) }^{\frac{p}{m}} \int_{f^{-1}(a,b)} {| \nabla f | }^p\ d\mu_g,$$ and for any regular value $a_0=f_*(r_0)$ of $f$
\begin{eqnarray}\label{itismyfault-1}
|\nabla f_*(r_0)|\leq  {\left( \frac{V_m}{V'} \right) }^{\frac{1}{m}}\max_{f^{-1}(a_0)}|\nabla f|
\end{eqnarray}
where $V'$ denotes $(\frac{k}{m-1})^{\frac{m}{2}}\ V$.
\end{Corollary}
\begin{proof}
Set $g':=\frac{k}{m-1}g$. The spherical rearrangement of $f$ on $(M,g')$ is the same function $f_*$ on the $m$-sphere but with metric $\textsl{g}_{V'}=\frac{k}{m-1}\textsl{g}_V$.
Since the Ricci curvature of $g'$ is not less than $m-1$, one can apply the above proposition to $(M,g')$ with volume $V'$ and $S^m_{V'}$ to get
\begin{eqnarray*}
\int_{f^{-1}(a,b)} | \nabla f |_g^p\ d\mu_g &=& \left(\frac{m-1}{k}\right)^{\frac{m-p}{2}} \int_{f^{-1}(a,b)} | \nabla f |_{g'}^p\ d\mu_{g'}\\
 &\geq&  \left(\frac{m-1}{k}\right)^{\frac{m-p}{2}} {\left( \frac{V'}{V_m} \right) }^{\frac{p}{m}}
\int_{f_*^{-1}(a,b)} | \nabla f_*|_{\textsl{g}_{V'}}^p d\mu_{\textsl{g}_{V'}}\\ &=&  {\left( \frac{V'}{V_m} \right) }^{\frac{p}{m}}
\int_{f_*^{-1}(a,b)} | \nabla f_*|_{\textsl{g}_{V}}^p d\mu_{\textsl{g}_{V}},
\end{eqnarray*}
and the 2nd inequality is obtained in a similar way or from the 1st inequality by letting $b\rightarrow a$.
\end{proof}

\subsection{Fiberwise spherical rearrangement}

Continuing the above discussion, we extend it to the fiberwise spherical rearrangement in a Riemannian product $(N\times M, h+g)$ where $(N,h)$ is a smooth closed Riemannian manifold as a parameter space. For a smooth generic-fiberwise Morse function $F :N\times M \rightarrow \Bbb R$, we let $N_0=N_0(F)$ be the open dense subset of $N$ such that $F_s$ for any $s\in N_0$ is a Morse function on $M$, and define $$F_{\bar{*}}:N_0\times S^m_V\rightarrow \Bbb R$$ as the fiberwise spherical rearrangement of $F$; namely, $F_{\bar{*}s}:=(F_{\bar{*}})_s$ for any $s\in N_0$ is equal to $(F_s)_{*}$.

For now $F_{\bar{*}}$ is defined only for $N_0\times S^m_V$, but we shall later extend it to a continuous function on $N\times S^m_V$.
As before, $F_{\bar{*}s}: S^m_V\rightarrow \Bbb R$ for each $s\in N_0$ is radially symmetric, and  we regard $F_{\bar{*}}$ also as a function on $N_0\times [0, R_V]$ by abuse of notation.

Let's first show that $F_{\bar{*}s}$ is smooth as far as $F_s$ is regular.
To do it, we need 3 functions $$\bar{F} :N\times M \rightarrow N\times \Bbb R,$$  $$\mathcal{V}:N_0\times \Bbb R\rightarrow \Bbb R,$$ $$\bar{\mathcal{V}}:N_0\times \Bbb R\rightarrow N_0\times \Bbb R$$
defined by $$\bar{F}(s,x):=(s,F(s,x)),\ \ \ \mathcal{V}(s,t):=\mu(\{p\in M| F(s,p)<t \}),\ \ \ \bar{\mathcal{V}}(s,t):=(s,\mathcal{V}(s,t))$$
respectively. Since $F_s$ for $s\in N_0$ is a Morse function on $M$, $\mathcal{V}(s,t)$ is the same as $\mu(\{p\in M| F(s,p)\leq t \})$ and $\mathcal{V}$ restricted to each $\{s\}\times \textrm{Im}(F_s)$ is a strictly increasing function, so $\bar{\mathcal{V}}$ is 1-1 on $\cup_{s\in N_0}(\{s\}\times \textrm{Im}(F_s))$ allowing $\bar{\mathcal{V}}^{-1}$.

Let $A(r)$ for $r\in [0, R_V]$ be the volume of a geodesic ball of radius $r$ in $S^m_V$. Obviously $A$ is a continuous function which is also smooth and regular on $(0, R_V)$ as seen in the explicit computation of $A(r)$.(In fact, $A'(r)$ is given by the $(m-1)$-dimensional volume of a geodesic sphere of radius $r$ in $S^m_V$.)  By our definitions,
\begin{eqnarray}\label{BJY}
A(r)=\mathcal{V}(s,F_{\bar{*}}(s,r))
\end{eqnarray}
for any $s\in N_0$ and $r\in [0,R_V]$.

\begin{Proposition}
$\mathcal{V}$ is continuous on $N_0\times \Bbb R$.
\end{Proposition}
\begin{proof}
Let's show the continuity of $\mathcal{V}$ at $(s_0,t_0)\in N_0\times \Bbb R$.
Since $F_{s_0}$ is a Morse function on $M$, an increasing function $\mathcal{V}_{s_0}$ on $\{s_0\}\times \textrm{Im}(F_{s_0})$ cannot have a jump discontinuity, and hence it is continuous. So for any $\epsilon>0$ there exists $\tau>0$ such that $$|\mathcal{V}(s_0,t_0+\tau)-\mathcal{V}(s_0,t_0-\tau)|<\epsilon.$$
By the uniform continuity of $F$, there also exists $\delta>0$ such that $$\textrm{dist}_{N\times M}((s,x),(s',x'))<\delta \ \ \ \Rightarrow  \ \ \ |F(s,x)-F(s',x')|<\frac{\tau}{2}.$$ We denote the distance function on a Riemannian manifold $X$ by $\textrm{dist}_X$.

Now if $|t-t_0|<\frac{\tau}{2}$ and $\textrm{dist}_{N}(s,s_0)<\delta$, then
\begin{eqnarray*}
\{p\in M|F(s_0,p)<t_0-\tau\} &\subseteq& \{p\in M|F(s_0,p)<t-\frac{\tau}{2}\} \\ &\subseteq& \{p\in M|F(s,p)<t\} \\ &\subseteq& \{p\in M|F(s_0,p)<t+\frac{\tau}{2}\} \\ &\subseteq& \{p\in M|F(s_0,p)<t_0+\tau\}
\end{eqnarray*}
so that both of $\{p\in M|F(s,p)<t\}$ and $\{p\in M|F(s_0,p)<t_0\}$  are contained in $\{p\in M|F(s_0,p)<t_0+\tau\}$ and contain $\{p\in M|F(s_0,p)<t_0-\tau\}$, and hence
$$|\mathcal{V}(s,t)-\mathcal{V}(s_0,t_0)|<\mathcal{V}(s_0,t_0+\tau)-\mathcal{V}(s_0,t_0-\tau)<\epsilon.$$ So the continuity of $\mathcal{V}$ is proved.
\end{proof}

\begin{Lemma}\label{hyunkeun}
Suppose $t_0\in \Bbb R$ is a regular value of $F_{s_0}$ for $s_0\in N_0$. Then $\mathcal{V}$ is smooth around $(s_0,t_0)$ and $\mathcal{V}(s_0,t_0)$ is a regular value of $\mathcal{V}_{s_0}$.
\end{Lemma}
\begin{proof}
Let us first prove the 2nd statement which is easier.
To show the regularity of $\mathcal{V}$, we use the coarea formula to rewrite
\begin{eqnarray}\label{Hitchin-Seminar}
\mathcal{V}(s,t)=\mu(F^{-1}_s(-\infty,t))=\int_{-\infty}^t\int_{F^{-1}_s(\tau)}|\nabla^M F_s|^{-1}d\sigma_{\tau}d\tau
\end{eqnarray}
where $\nabla^M$ means the fiberwise gradient on $\{s\}\times M$. Since $t_0$ is a regular value of $F_{s_0}$ and $M$ is compact, there exists an $\varepsilon>0$ such that $(t_0-\varepsilon, t_0+\varepsilon)$ does not contain a critical value of $F_{s_0}$. So we can safely take the derivative
\begin{eqnarray*}
\frac{d}{dt}|_{t=t_0}\mathcal{V}(s_0,t)&=& \frac{d}{dt}|_{t=t_0}\int_{t_0-\varepsilon}^t\int_{F^{-1}_{s_0}(\tau)}|\nabla^M F_{s_0}|^{-1}d\sigma_{\tau}d\tau\\ &=& \int_{F^{-1}_{s_0}(t_0)}|\nabla^M F_{s_0}|^{-1}d\sigma_{t_0}>0,
\end{eqnarray*}
thereby proving that $t_0$ is regular point of $\mathcal{V}_{s_0}$.

The 1st statement is not obtained in a similar way.
It is not obvious whether one can simply take the derivative of (\ref{Hitchin-Seminar}) with respect to $s$ due to the presence of zeros of $|\nabla^M F_s|$ in the integration domain. Moreover the zero set is also changing as $s$ changes. To avoid such complications, we take a different approach.

Since $(s_0,t_0)$ is a regular value of $\bar{F}$, and $N\times M$ is compact, there exists its open neighborhood consisting of regular values of $\bar{F}$, and hence each $\bar{F}^{-1}((s,t))$ for $(s,t)$ near $(s_0,t_0)$ in $N_0\times \Bbb R$ is an embedded submanifold (without boundary). The key idea for proof is that one can locally straighten out each $\bar{F}^{-1}((s,t))$ simultaneously in some local chart so that $\bar{F}$ looks like a standard projection map.
To show the smoothness of $\mathcal{V}$, we may identify a neighborhood of $s_0$ in $N_0$ with a subset $\mathcal{U}$ in $\Bbb R^n$ (via a local chart).
\begin{Sublemma}\label{Kangkyungwha}
For any $p\in \bar{F}^{-1}((0,t_0))$ we can take a coordinate neighborhood $U_p\subset \mathcal{U}\times M$ of $p$ with a coordinate map $\psi_p:U_p\rightarrow \Bbb R^{n+m}$ such that  $$\bar{F}\circ\psi_p^{-1}:\psi_p(U_p)\rightarrow  \mathcal{U}\times (t_0-\epsilon,t_0+\epsilon)$$ is the standard projection map $\Pi_{n+1}:\Bbb R^{n+m}\rightarrow \Bbb R^{n+1}$ onto the first $n+1$ coordinates.
\end{Sublemma}
\begin{proof}
This is the consequence of the canonical submersion theorem whose proof can be found in \cite{pollack}.


\end{proof}

Now for any $q\in N\times M-\bar{F}^{-1}((s_0,t_0))$, we take an open neighborhood $U_q\subset N\times M$ of $q$  such that its closure $\overline{U_q}$ satisfies
\begin{eqnarray*}
\overline{U_q}\cap \bar{F}^{-1}((s_0,t_0))=\emptyset.
\end{eqnarray*}

The union of all such $U_q$'s and also $U_p$'s for $p\in\bar{F}^{-1}((s_0,t_0))$ as in the above sublemma gives an open cover of a compact manifold $N\times M$.
Let $\cup_{i=1}^lU_i$ be its finite subcover indexed such that $U_i$ for $i=1,\cdots,l'$ is of $U_p$-type and  $U_i$ for $i=l'+1,\cdots,l$ is of $U_q$-type.

Take a smooth partition $\{\rho_i|i=1,\cdots,l\}$ of unity subordinate to $\{U_i|i=1,\cdots,l\}$ where each $\textrm{supp}(\rho_i)$ is compact.
By using $\sum_{i=1}^l\rho_i\equiv1$,
\begin{eqnarray*}
\mathcal{V}(s,t) &=& \int_{F_s^{-1}(-\infty,t)}d\mu_g \\ &=& \int_{F_s^{-1}(-\infty,t)}\sum_{i=1}^l\rho_i\ d\mu_g \\&=& \sum_{i=1}^l\int_{U_i\cap F_s^{-1}(-\infty,t)}\rho_i\ d\mu_g\\ &=&
\sum_{i=1}^l\mathcal{V}_i(s,t)
\end{eqnarray*}
where the last equality just states the definition of $\mathcal{V}_i(s,t)$.

Since $(s_0,t_0)$ does not belong to a compact subset $\bar{F}(\cup_{i=l'+1}^l\overline{U_i})$, there exists an open neighborhood of $(s_0,t_0)$, with which $\bar{F}(\cup_{i=l'+1}^l\overline{U_i})$ has no intersection.
Thus $\mathcal{V}_i$ for $i=l'+1,\cdots,l$ is constant in the neighborhood, and hence to show the smoothness of $\mathcal{V}$ around $(s_0,t_0)$, it is enough to show that each $\mathcal{V}_i$ for $i=1,\cdots,l'$ is smooth there.

To do it, we express each $\mathcal{V}_i$ for $i=1,\cdots,l'$ using the local coordinate $$\psi_i(U_i)=\{(x_1,\cdots,x_{n+m})\in \Bbb R^{n+m}|a_j<x_j<b_j, \forall j\}$$ given in the above sublemma.
Indeed  $\mathcal{V}_i((x_1,\cdots,x_n),t)$ is written as
$$\int_{a_{n+m}}^{b_{n+m}}\cdots\int_{a_{n+2}}^{b_{n+2}}\int_{a_{n+1}}^{t}\rho_i\circ \psi_i^{-1}(x_1,\cdots,x_{n+m})\ \xi_i(x_1,\cdots,x_{n+m})\ dx_{n+1}dx_{n+2}\cdots dx_{n+m}$$
where $\xi_i(x_{1},\cdots,x_{n+m})dx_{1}\cdots dx_{n+m}$ is the volume element of $(N\times M,h+g)$ expressed in $\psi_i(U_i)$. Because $\rho_i$ is compact-supported in $U_i$, one can freely differentiate $\mathcal{V}_i$ arbitrarily times, that is, $\mathcal{V}_i$ is a smooth function of $(x_1,\cdots,x_n,t)\in \prod_{j=1}^{n+1}(a_{j},b_{j})$.
\end{proof}

\begin{Proposition}\label{mysister}
Suppose $t_0\in \Bbb R$ is a regular value of $F_{s_0}$ and let $F_{\bar{*}}(s_0,r_0)=t_0$ for $r_0\in (0,R_V)$. Then $F_{\bar{*}}$ is smooth in an open neighborhood of $(s_0,r_0)$ such that any point $(s,r)$ in the neighborhood is a regular point of $F_{\bar{*}s}=F_{\bar{*}}|_{\{s\}\times [0,R_V]}$.
\end{Proposition}
\begin{proof}
By the previous lemma, $(s_0,t_0)$ is a regular point of $\bar{\mathcal{V}}$ so that $\bar{\mathcal{V}}$ is a local diffeomorphism around $(s_0,t_0)$.
From the relation $$\bar{\mathcal{V}}(s, F_{\bar{*}}(s,r))=(s,A(r)),$$ it follows that
$$(s,F_{\bar{*}}(s,r))=\bar{\mathcal{V}}^{-1}((s,A(r)))$$ is a smooth function of $(s,r)$ around $(s_0,r_0)$ and $F_{\bar{*}s_0}$ is regular at $r_0$ by using the regularity of $A$ and $\bar{\mathcal{V}}^{-1}$. The regularity at one point implies the regularity in a neighborhood.
\end{proof}

Since $F_s$ for any $s\in N_0$ is a Morse function on $M$, the implicit function theorem guarantees that the critical locus
\begin{eqnarray}\label{LHW}
\mathcal{C}:=\{(s,x)\in N_0\times M| \nabla^M F_s(x)=0\}
\end{eqnarray}
where $\nabla^M$ is the fiberwise gradient on each $\{pt\}\times M$ is an $n$-dimensional submanifold of $N_0\times M$ such that $\mathcal{C}\cap (\{s\}\times M)$ for each $s\in N_0$ consists of finite points. Considering the differential $d\bar{F}:T(N\times M)\rightarrow T(N\times \Bbb R)$, the kernel of $d\bar{F}|_{T\mathcal{C}}$ is zero and hence $\bar{F}|_{\mathcal{C}}$ is an immersion. Thus  $\bar{F}(\mathcal{C})$ is an immersed hyupersurface of $N_0\times \Bbb R$ such that $\bar{F}(\mathcal{C})\cap (\{s\}\times \Bbb R)$ for each $s\in N_0$ consists of finite points.

\begin{Lemma}
$\bar{F}(\mathcal{C})$ is a relatively-closed subset of $N_0\times \Bbb R$.
\end{Lemma}
\begin{proof}
We show that $N_0\times \Bbb R-\bar{F}(\mathcal{C})$ is open.
For any $s\in N_0$, let's take an open neighborhood $U$ of $s$ whose closure $\overline{U}$ is contained in $N_0$. Since $\mathcal{C}$ which is the common zero set of $m$ smooth functions is a relatively-closed subset of $N_0\times M$,  $\mathcal{C}\cap (\overline{U}\times M)$ is compact. So $\bar{F}(\mathcal{C}\cap (\overline{U}\times M))=\bar{F}(\mathcal{C})\cap (\overline{U}\times \Bbb R)$ is compact and hence closed. Now
$$U\times \Bbb R-\bar{F}(\mathcal{C})=U\times \Bbb R-(\bar{F}(\mathcal{C})\cap (\overline{U}\times \Bbb R))$$ is open, by which the desired conclusion follows.
\end{proof}

We define $$\mathcal{C}^*:=\{(s,y)\in N_0\times S^m_V|(s,F_{\bar{*}}(s,y))\in \bar{F}(\mathcal{C})\}$$ and
use the notation $$U(r',r''):=\{y\in S^m_V|r'< r(y)< r''\}$$ for any $r',r''\in \Bbb R$.
By Proposition \ref{mysister}, $F_{\bar{*}}$ is smooth on $N_0\times S^m_V-\mathcal{C}^*$, and moreover we have :
\begin{Proposition}
$F_{\bar{*}}:N_0\times S^m_V\rightarrow \Bbb R$ is continuous, and $\mathcal{C}^*$ is a relatively-closed measure-zero subset of $N_0\times S^m_V$.
\end{Proposition}
\begin{proof}
To show that $F_{\bar{*}}^{-1}(-\infty, \alpha)$ for any $\alpha\in \Bbb R$ is an open subset of $N_0\times S^m_V$,
let $p\in F_{\bar{*}}^{-1}(-\infty, \alpha)$.

If $p\in N_0\times S^m_V-\mathcal{C}^*$, then by Proposition \ref{mysister} there exists an open neighborhood $U(p)\subset N_0\times S^m_V$ of $p$ such that $F_{\bar{*}}$ is smooth on $U(p)$. So by taking sufficiently small $U(p)$, $F_{\bar{*}}(U(p))\subset (-\infty, \alpha)$.

If $p=(s_0,y_0)\in \mathcal{C}^*$ and $F_{\bar{*}}(p)$ is not the maximum of $F_{\bar{*}s_0}$, then by the continuity of $F_{\bar{*}s_0}$ one can take $p'=(s_0,y_0')\in F_{\bar{*}}^{-1}(-\infty, \alpha)$ satisfying $r(y_0')>r(y_0)$ and $p'\notin \mathcal{C}^*$. By the above case, we can take a neighborhood $$U(p'):=B_\epsilon(s_0)\times U(r(y_0')-\epsilon,r(y_0')+\epsilon)$$ for some $\epsilon>0$ such that $F_{\bar{*}}(U(p'))\subset  (-\infty, \alpha)$. Then $$F_{\bar{*}}(B_\epsilon(s_0)\times U(r(y_0)-\epsilon,r(y_0)+\epsilon))\subset  (-\infty, \alpha)$$ by the monotone-increasing property of $F_{\bar{*}s}$ for each $s$.

If $p=(s_0,y_0)\in \mathcal{C}^*$ and $F_{\bar{*}}(p)$ is the maximum of $F_{\bar{*}s_0}$, then it must be that there exists $\varepsilon>0$ such that $\alpha> \max F_{s}$ for any $s\in B_\varepsilon(s_0)$ by the continuity of $F$.
So for an open neighborhood $B_\varepsilon(s_0)\times S^m_V$ of $p$,  $$F_{\bar{*}}(B_\varepsilon(s_0)\times S^m_V)\subset (-\infty, \alpha).$$

In the similar way it is proved that
$F_{\bar{*}}^{-1}(\alpha,\infty)$ for any $\alpha\in \Bbb R$ is open.(In this case, for $p\in F_{\bar{*}}^{-1}(\alpha,\infty)$ one has to consider whether  $F_{\bar{*}}(p)$ is the minimum of $F_{\bar{*}s_0}$ or not.)  This finishes the proof of the continuity of $F_{\bar{*}}$, and hence $\mathcal{C}^*$ which is the inverse image of a relatively-closed subset $\bar{F}(\mathcal{C})$ of $N_0\times \Bbb R$ under a continuous map is relatively closed in $N_0\times S^m_V$.

The measure of $\mathcal{C}^*$ is computed as
\begin{eqnarray*}
\int\int_{N_0\times S^m_V}\chi_{_{\mathcal{C}^*}}\ d\mu_{h+\textsl{g}_V} &=& \int\int_{N_0\times S^m_V}\chi_{_{\mathcal{C}^*}}\ d(\mu_h\times \mu_{\textsl{g}_V})\\ &=& \int_{N_0}\left(\int_{S^m_V}\chi_{_{\mathcal{C}^*}}\ d\mu_{\textsl{g}_V}\right) d\mu_h\\ &=& \int_{N_0}0\ d\mu_h\\ &=& 0
\end{eqnarray*}
by using the fact that $\mathcal{C}^*\cap(\{s\}\times S^m_V)$ for each $s\in N_0$ is the union of a finite number of geodesic spheres and the two poles.

\end{proof}

Thus $N_0\times S^m_V-\mathcal{C}^*$ is dense in $N_0\times S^m_V$ and hence in $N\times S^m_V$. Summing up, we have proved that $F_{\bar{*}}$ is smooth on an open dense subset with measure-zero complement in $N\times S^m_V$. Next, we shall show that $F_{\bar{*}}$ can be extended to $N\times S^m_V$ continuously
by using the derivative estimate along $S^m_V$ direction obtained in Subsection 4.1.
\begin{Proposition}\label{hanguk}
Let $(M^m,g)$ with volume $V$ and $(N^n,h)$ be smooth closed Riemannian manifolds. Suppose that $F:N\times M\stackrel{C^\infty}{\rightarrow} \Bbb R$ is a generic-fiberwise Morse function. Then its fiberwise rearrangement $F_{\bar{*}}:N_0\times S^m_V\rightarrow \Bbb R$ extends to a Lipschitz continuous function on $N\times S^m_V$ such that it is radially symmetric in each $\{pt\}\times S^m_V$, and such an extension is unique independent of the choice of $N_0$.

In addition, if a Lie group $G$ acts on $(M,g)$ isometrically and on $N$ both by smooth right actions, then for any $\mathfrak{g}\in G$,
\begin{eqnarray}\label{Ocean-nuclear}
(F\cdot \mathfrak{g})_{\bar{*}}=F_{\bar{*}}\cdot \mathfrak{g}
\end{eqnarray}
where the induced $G$ action on $N\times S^m_V$ acts trivially on $S^m_V$. The analogous statement holds for a left action too.
\end{Proposition}
\begin{proof}
Since $N_0\times S^m_V$ is a dense subset, it's enough for the existence of a unique Lipschitz continuous extension to show that $F_{\bar{*}}$ is Lipschitz continuous on $N_0\times S^m_V$.
We already know that for each $s\in N_0$, a continuous function $F_{\bar{*}s}:S^m_V\rightarrow \Bbb R$ is piecewise-smooth and hence Lipschitz continuous. By Proposition \ref{sejong-brt}, a finite number $$C_1':=C_1\max \{|\nabla^M F_s(x)| \ | \ (s,x)\in N\times M\}$$ upper-bounds all the Lipschitz constants of $F_{\bar{*}s}$ for $s\in N_0$.

On the other hand, by Proposition \ref{Firstlady-1} we have
$$||F_{\bar{*}s}-F_{\bar{*}t}||_\infty\leq ||F_s-F_t||_\infty$$ for any $s,t\in N_0$.

Since $F$ is Lipschitz continuous on $N\times M$, there exists a constant $C_2>0$ such that for any $s,t\in N$
\begin{eqnarray*}
||F_s-F_t||_\infty &=& \sup_{x\in M} |F_s(x)-F_t(x)| \\ &\leq& \sup_{x\in M} C_2d_N(s,t) \\ &=& C_2 d_N(s,t)
\end{eqnarray*}
 where $d_N$ is the induced distance function on $(N,h)$.

Combining all the above, we deduce that for any $s,t\in N_0$ and $y,z\in S^m_V$,
\begin{eqnarray*}
|F_{\bar{*}}(s,y)-F_{\bar{*}}(t,z)| &\leq& |F_{\bar{*}}(s,y)-F_{\bar{*}}(t,y)|+|F_{\bar{*}}(t,y)-F_{\bar{*}}(t,z)| \\
&\leq& ||F_s-F_t||_\infty+C_1'd_S(y,z)\\  &\leq&  C_2 d_N(s,t)+C_1'd_S(y,z)\\
&\leq& \max(C_1,C_2)\sqrt{2}\ d((s,y),(t,z))
\end{eqnarray*}
where $d_S$ denotes the induced distance function on $(S^m_V,\textsl{g}_V)$ and $d$ is the usual product distance function on $N\times S^m_V$ given by $\sqrt{(d_N\circ \pi_1)^2+(d_S\circ \pi_2)^2}$.

To show the uniqueness independent of the choice of $N_0$, let $\tilde{N}_0$ be another open dense subset with measure-zero complement for which $F:\tilde{N}_0\times M\rightarrow \Bbb R$ is a fiberwise Morse function. By Baire's category theorem, $N_0\cap\tilde{N}_0$ is also an open dense subset with measure-zero complement. Therefore $F_{\bar{*}}:(N_0\cap\tilde{N}_0)\times S^m_V\rightarrow \Bbb R$ must extend to the above $F_{\bar{*}}:N\times S^m_V\rightarrow \Bbb R$.

The fiberwise radial symmetry of $F_{\bar{*}}$ is obvious, because it already holds on a dense subset $N_0\times S^m_V$.

Now we prove (\ref{Ocean-nuclear}).
For any $s\in N_0\cdot \mathfrak{g}$, as a function on $S^m_V$
$$(F\cdot\mathfrak{g})_{\bar{*}s}=((F\cdot\mathfrak{g})_{s})_*=(F_{s\cdot \mathfrak{g}^{-1}}\cdot \mathfrak{g})_*=(F_{s\cdot \mathfrak{g}^{-1}})_*=(F_{\bar{*}}\cdot \mathfrak{g})_{s}$$ where the 3rd equality is due to the fact that the $G$-action on $(M,g)$ is isometric.
Thus for any $\mathfrak{g}\in G$, (\ref{Ocean-nuclear}) holds on $(N_0\cdot \mathfrak{g})\times S^m_V$ and hence on $N\times S^m_V$ by the denseness of $N_0$.
The case of a left action is proved in the same way. For any $\mathfrak{g}\in G$ and  $s\in \mathfrak{g}\cdot N_0$, as a function on $S^m_V$
$$(\mathfrak{g}\cdot F)_{\bar{*}s}=((\mathfrak{g}\cdot F)_{s})_*=(\mathfrak{g}\cdot F_{\mathfrak{g}^{-1}\cdot s})_*=(F_{\mathfrak{g}^{-1}\cdot s})_*=(\mathfrak{g}\cdot F_{\bar{*}})_{s},$$ so
\begin{eqnarray*}
(\mathfrak{g}\cdot F)_{\bar{*}}=\mathfrak{g}\cdot F_{\bar{*}}
\end{eqnarray*}
holds on $(\mathfrak{g}\cdot N_0)\times S^m_V$ and hence $N\times S^m_V$.
\end{proof}


\subsection{Derivative estimates of $F_{\bar{*}}$}

We know that $F_{\bar{*}}$ is smooth a.e. namely in the open subset $N_0\times S^m_V-\mathcal{C}^*$. Now let's estimate its strong, i.e ordinary 1st order derivatives defined there. We shall use it to show that $F_{\bar{*}}$ is weakly differentiable on $N\times S^m_V$.
The derivative along $S^m_V$ direction was already estimated in Subsection 4.1 and here we proceed to the integral estimate of the derivative of $F_{\bar{*}}$ along $N_0$ direction.

For a local coordinate $(x_1,\cdots,x_n)$ of $N_0$, to compute $\frac{\partial F_{\bar{*}}}{\partial x_j}$, we may assume that $N_0$ is an interval $(s_1,s_2)$ in the $j$-th coordinate line for convenience.
Let $s_0\in N_0$. For any two regular values $a_1<a_2$ of $F_{s_0}$, let $r_i$ for $i=1,2$ be the real number satisfying $$F_{\bar{*}}(s_0,r_i)=a_i,$$
and define $\alpha_i:(s_1,s_2)\rightarrow \Bbb R$ by $$\alpha_i(s):=F_{\bar{*}}(s,r_i),$$
and
$$\chi_{s}:M\rightarrow [0,1]\ \ \ \ \textrm{and}\ \ \ \ \chi^*:S^m_V\rightarrow [0,1]$$ by
the characteristic functions of $$\{x\in M|\alpha_1(s)< F(s,x)< \alpha_2(s)\}\ \ \  \textrm{and}\  \ \
U(r_1,r_2)$$ respectively.
Note that $\chi_s$ and $\chi^*$ are defined so that for any $s$
\begin{eqnarray}\label{gguggu-0}
\int_{M} \chi_{s}\ d\mu_g=\int_{S^m_V} \chi^*\ d\mu_{\textsl{g}_V},
\end{eqnarray}
and
\begin{eqnarray}\label{gguggu}
\int_{M} F_s\chi_{s}\ d\mu_g=\int_{S^m_V} F_{\bar{*}s}\chi^*\ d\mu_{\textsl{g}_V}.
\end{eqnarray}
Lastly let us fix some notations :
$$D_\tau:= \cup_{i=1}^2\{x\in M|\ |F(s_0,x)-a_i|\leq \tau\},$$
$$\Sigma_{s,t}:=\{x\in M|F(s,x)=t\}\ \ \ \ \textrm{and}\ \ \ \  \Sigma^*_{s,t}:=\{y\in S^m_V|F_{\bar{*}}(s,y)=t\}$$ for any $\tau\geq 0$, $t\in \Bbb R$ and $s\in N_0$.

The fundamental lemma of this subsection is :
\begin{Lemma}
\begin{eqnarray}\label{officetel-1}
\frac{\partial }{\partial s}|_{s=s_0}\int_{M} F(s,x)\chi_{s}(x)\ d\mu_g(x) &=& \int_M \frac{\partial F}{\partial s}(s_0,x)\chi_{s_0}(x)\ d\mu_g(x).
\end{eqnarray}
If $[a_1,a_2]$ does not contain a critical value of $F_{s_0}$, then
\begin{eqnarray}\label{officetel-2}
\frac{\partial }{\partial s}|_{s=s_0}\int_{S^m_V} F_{\bar{*}}(s,y)\chi^*(y)\ d\mu_{\textsl{g}_V}(y)
&=&\int_{S^m_V} \frac{\partial F_{\bar{*}}}{\partial s}(s_0,y)\chi^*(y)\ d\mu_{\textsl{g}_V}(y).
\end{eqnarray}
\end{Lemma}
\begin{proof}
Since $(s_0,a_1),(s_0,a_2)\in N\times \Bbb R$ are regular values of $\bar{F}$, we can choose a positive constant $\varepsilon<\frac{a_2-a_1}{2}$ such that $[s_0-\varepsilon,s_0+\varepsilon]\times [a_i-\varepsilon,a_i+\varepsilon]$ for $i=1,2$ does not contain a critical value of $\bar{F}$.
By the continuity of $\alpha_i$ we can choose  $\delta\in (0,\varepsilon)$ such that
\begin{eqnarray}\label{vietnam}
\alpha_i([s_0-\delta,s_0+\delta])\subset (a_i-\varepsilon,a_i+\varepsilon)
\end{eqnarray}
for $i=1,2$, so each $\alpha_i$ is differentiable at any point of $[s_0-\delta,s_0+\delta]$.

Denote $M_1:=\{x\in M|F(s_0,x)<\frac{a_1+a_2}{2}\}$ and $M_2:=M-M_1$, and we compute $\frac{\partial }{\partial s}|_{s=s_0}\int_{M} F_s\chi_{s}\ d\mu_g$ as
\begin{eqnarray*}
 & & \lim_{\Delta s\rightarrow 0}\frac{\int_M(F(s_0+\Delta s,x)\chi_{s_0+\Delta s}(x)-F(s_0,x)\chi_{s_0}(x))\ d\mu_g(x)}{\Delta s}\\
 &=& \lim_{\Delta s\rightarrow 0}\frac{\int_{M}[(F(s_0+\Delta s,x)-F(s_0,x))\chi_{s_0+\Delta s}+F(s_0,x)(\chi_{s_0+\Delta s}-\chi_{s_0})]
\ d\mu_g}{\Delta s}\\
&=& \lim_{\Delta s\rightarrow 0}\frac{\int_M(F(s_0+\Delta s,x)-F(s_0,x))\chi_{s_0+\Delta s}(x)\ d\mu_g(x)}{\Delta s}+\\ & & \lim_{\Delta s\rightarrow 0}\sum_{i=1}^2(\int_{M_i} a_i\frac{\chi_{s_0+\Delta s}-\chi_{s_0}}{\Delta s}\ d\mu_g+\int_{M_i} (F(s_0,x)-a_i)\frac{\chi_{s_0+\Delta s}-\chi_{s_0}}{\Delta s}\ d\mu_g)\\
&=&  \int_M \frac{\partial F}{\partial s}(s_0,x)\chi_{s_0}(x)\ d\mu_g(x),
\end{eqnarray*}
where the 2nd and 3rd equalities are due to the following 3 sublemmas.

\begin{Sublemma}\label{dr-carter-1}
$$\lim_{\Delta s\rightarrow 0}\frac{\int_M(F(s_0+\Delta s,x)-F(s_0,x))\chi_{s_0+\Delta s}(x)\ d\mu_g(x)}{\Delta s}=\int_M \frac{\partial F}{\partial s}(s_0,x)\chi_{s_0}(x)\ d\mu_g(x).$$
\end{Sublemma}
\begin{proof}
It is enough to show that for any sequence $\sigma_n\in (-\delta,0)\cup (0,\delta)$ converging to 0,
$$\lim_{n\rightarrow \infty}\frac{\int_M(F(s_0+\sigma_n,x)-F(s_0,x))\chi_{s_0+\sigma_n}(x)\ d\mu_g(x)}{\sigma_n}= \int_M \frac{\partial F}{\partial s}(s_0,x)\chi_{s_0}(x)\ d\mu_g(x).$$


Since $F(s,x)-\alpha_i(s)$ for any $x\in M$ and $i=1,2$ is a continuous function of $s$, $$\lim_{\Delta s\rightarrow 0}\chi_{s_0+\Delta s}(x)=\chi_{s_0}(x)$$ for any $x\in M-F^{-1}_{s_0}\{\alpha_1(s_0),\alpha_2(s_0)\}$ and hence almost every $x$ in $M$, because a smooth compact hypersurface $F^{-1}_{s_0}\{\alpha_1(s_0),\alpha_2(s_0)\}$ is a measure-zero subset of $M$.

By the mean value theorem, for any $x\in M$, $n\in \Bbb N$ there exists $\theta_n(x)\in (0,1)$ such that
\begin{eqnarray*}
\left|\frac{F(s_0+\sigma_n,x)-F(s_0,x)}{\sigma_n}\right| &=& \left|\frac{\partial F}{\partial s}(s_0+\theta_n(x)\sigma_n,x)\right| \\ &\leq&
\sup_{[s_0-\delta,s_0+\delta]\times M}\left|\frac{\partial F}{\partial s}\right| \\ &<& \infty.
\end{eqnarray*}
Thus one can apply Lebesgue's dominated convergence theorem to
$$
\lim_{n\rightarrow \infty}\frac{\int_M(F(s_0+\sigma_n,x)-F(s_0,x))\chi_{s_0+\sigma_n}(x)\ d\mu_g(x)}{\sigma_n}
$$
to obtain the desired equality.

\end{proof}

\begin{Sublemma}\label{dr-carter-1.5}
$$\lim_{\Delta s\rightarrow 0}\int_{M_i} \frac{\chi_{s_0+\Delta s}-\chi_{s_0}}{\Delta s}\ d\mu_g=0$$  for $i=1,2$.
\end{Sublemma}
\begin{proof}
This is also easily proved in the light of (\ref{gguggu-0}).
Note that $\bigcup_{s\in (s_0-\delta, s_0+\delta)}\Sigma_{s,\alpha_i(s)}$ for $i=1,2$ is a smooth hypersurface of $(s_0-\delta, s_0+\delta)\times M$. Since $\Sigma_{s_0,\alpha_1(s_0)}\subseteq M_1$ and $\Sigma_{s_0,\alpha_2(s_0)}\subseteq \textrm{int}(M_2)$ are compact hypersurfaces of $M$, there exists $\delta'\in (0,\delta)$ such that $\Sigma_{s,\alpha_1(s)}\subseteq M_1$ and $\Sigma_{s,\alpha_2(s)}\subseteq \textrm{int}(M_2)$ for any $s\in [s_0-\delta', s_0+\delta']$.

Therefore
\begin{eqnarray*}
\lim_{\Delta s\rightarrow 0}\int_{M_1} \frac{\chi_{s_0+\Delta s}-\chi_{s_0}}{\Delta s}\ d\mu_g &=& \frac{\partial}{\partial s}|_{s=s_0}\mu(\{x\in M_1|\alpha_1(s)<F(s,x)<\alpha_2(s)\})\\ &=&
\frac{\partial}{\partial s}|_{s=s_0}\mu(M_1-\{x\in M|F(s,x)\leq \alpha_1(s)\})\\ &=&
\frac{\partial}{\partial s}|_{s=s_0}\left(\mu(M_1)-\mathcal{V}(s,\alpha_1(s))\right)\\ &=&
-\frac{\partial}{\partial s}|_{s=s_0}\mathcal{V}(s,\alpha_1(s))\\ &=& -\frac{\partial}{\partial s}|_{s=s_0}A(r_1)\\ &=& 0
\end{eqnarray*}
where the 5th equality is due to (\ref{BJY}).

The other part is obtained likewise as
\begin{eqnarray*}
\lim_{\Delta s\rightarrow 0}\int_{M_2} \frac{\chi_{s_0+\Delta s}-\chi_{s_0}}{\Delta s}\ d\mu_g &=&
\frac{\partial}{\partial s}|_{s=s_0}\mu(\{x\in M_2|\alpha_1(s)<F(s,x)<\alpha_2(s)\})\\ &=&
\frac{\partial}{\partial s}|_{s=s_0}\mu(\{x\in M|F(s,x)< \alpha_2(s)\}-M_1)\\ &=&
\frac{\partial}{\partial s}|_{s=s_0}(\mathcal{V}(s,\alpha_2(s))-\mu(M_1))\\ &=&
\frac{\partial}{\partial s}|_{s=s_0}\mathcal{V}(s,\alpha_2(s))\\ &=& \frac{\partial}{\partial s}|_{s=s_0}A(r_2)\\ &=& 0.
\end{eqnarray*}
\end{proof}

\begin{Sublemma}\label{dr-carter-2}
$$\lim_{\Delta s\rightarrow 0}\int_{M_i} (F(s_0,x)-a_i)\frac{\chi_{s_0+\Delta s}(x)-\chi_{s_0}(x)}{\Delta s}\ d\mu_g(x)=0$$ for $i=1,2$.
\end{Sublemma}
\begin{proof}
Set a constant $$\frak{m}:=\sup \{|\frac{\partial F}{\partial s}(s,x)|+|\alpha_i'(s)|\ |\ x\in M, s\in [s_0-\delta,s_0+\delta],  i=1,2\}.$$
By the mean value theorem,
$$F(s_0+\Delta s,x)=F(s_0,x)+\frac{\partial F}{\partial s}(s_0+\vartheta \Delta s,x)\Delta s$$ for some $\vartheta\in (0,1)$ depending on $x$ and $\Delta s$.
Thus for any $x\in M$ and $\Delta s\in (-\delta,\delta)$, $$|F(s_0+\Delta s,x)-F(s_0,x)|\leq  \frak{m}|\Delta s|$$ and likewise
$$|\alpha_i(s_0+\Delta s)-\alpha_i(s_0)|\leq \frak{m}|\Delta s|,$$ implying that
$$F(s_0,x)-\alpha_i(s_0)- 2\frak{m}|\Delta s|\leq F(s_0+\Delta s,x)-\alpha_i(s_0+\Delta s)\leq F(s_0,x)-\alpha_i(s_0)+ 2\frak{m}|\Delta s|.$$
If $x\in M-D_{2\frak{m}|\Delta s|}$, then $$|F(s_0,x)-\alpha_i(s_0)|> 2\frak{m}|\Delta s|$$ for $i=1,2$, and hence
the sign of $F(s_0+\Delta s,x)-\alpha_i(s_0+\Delta s)$ for $\Delta s\in (-\delta,\delta)$ is the same as that of $F(s_0,x)-\alpha_i(s_0)$, implying  $$\chi_{s_0+\Delta s}(x)=\chi_{s_0}(x).$$
Therefore for $\Delta s\in (-\delta,\delta)$ the support of $\chi_{s_0+\Delta s}-\chi_{s_0}$ is located in $D_{2\frak{m}|\Delta s|}$, and  for $|\Delta s|< \frac{1}{2\frak{m}}\frac{a_2-a_1}{2}$
$$M_i\cap D_{2\frak{m}|\Delta s|}=\{x\in M|\ |F(s_0,x)-a_i|\leq 2\frak{m}|\Delta s|\},$$
which enables us to estimate
\begin{eqnarray*}
\left|\int_{M_i} (F(s_0,x)-a_i)\frac{\chi_{s_0+\Delta s}-\chi_{s_0}}{\Delta s}\ d\mu_g\right| &\leq& \int_{M_i} |F(s_0,x)-a_i| \frac{|\chi_{s_0+\Delta s}-\chi_{s_0}|}{|\Delta s|}\ d\mu_g\\
&\leq& 2\frak{m}|\Delta s|\frac{1}{|\Delta s|}\mu(D_{2\frak{m}|\Delta s|})\\
&=& 2\frak{m}\ \mu(D_{2\frak{m}|\Delta s|})\\
&\rightarrow& 0
\end{eqnarray*}
as $\Delta s\rightarrow 0$ by using the fact that $D_\tau$ for any sufficiently small $\tau>0$ consists of thin tubular neighborhoods of smooth compact hypersurfaces $\Sigma_{s_0,a_1}$ and $\Sigma_{s_0,a_2}$ in $M$.
\end{proof}

The computation of $\frac{\partial }{\partial s}|_{s=s_0}\int_{S^m_V} F_{\bar{*}}\chi^*\ d\mu_{\textsl{g}_V}$ is easier, since $\chi^*$ does not depend on $s$. At this time by the additional condition that there are no critical values of $F_{s_0}$ in $[a_1,a_2]$, we choose $\varepsilon\in (0,\frac{a_2-a_1}{2})$ such that $[s_0-\varepsilon,s_0+\varepsilon]\times [a_1-\varepsilon,a_2+\varepsilon]$ does not contain a critical value of $\bar{F}$. For $\delta\in (0,\varepsilon)$ satisfying (\ref{vietnam}), i.e. $$F_{\bar{*}}([s_0-\delta,s_0+\delta]\times \overline{U(r_1,r_2)})\subset (a_1-\varepsilon,a_2+\varepsilon),$$ $F_{\bar{*}}$ is smooth on an open neighborhood of $[s_0-\delta,s_0+\delta]\times \overline{U(r_1,r_2)}$.
In particular $$\sup_{[s_0-\delta,s_0+\delta]\times \overline{U(r_1,r_2)}}\left|\frac{\partial F_{\bar{*}}}{\partial s}\right|<\infty,$$
so we can change the order of differentiation and integration in
\begin{eqnarray*}
\frac{\partial }{\partial s}|_{s=s_0}\int_{S^m_V} F_{\bar{*}}(s,y)\chi^*(y)\ d\mu_{\textsl{g}_V}(y)&=&
\frac{\partial }{\partial s}|_{s=s_0}\int_{U(r_1,r_2)} F_{\bar{*}}(s,y)\ d\mu_{\textsl{g}_V}(y)\\ &=&
\int_{U(r_1,r_2)} \frac{\partial F_{\bar{*}}}{\partial s}(s_0,y)\ d\mu_{\textsl{g}_V}(y)\\ &=&\int_{S^m_V} \frac{\partial F_{\bar{*}}}{\partial s}(s_0,y)\chi^*(y)\ d\mu_{\textsl{g}_V}(y).
\end{eqnarray*}

\end{proof}

\begin{Proposition}\label{blessed}
For $s_0\in N_0$ and a regular value $a_0=F_{\bar{*}}(s_0,r_0)$ of $F_{s_0}$
\begin{eqnarray}\label{Chokuk3}
\left|\frac{\partial F_{\bar{*}}}{\partial s}(s_0,r_0)\right|\leq \max_{\{x\in M|F(s_0,x)=a_0\}}\left|\frac{\partial F}{\partial s}(s_0,x)\right|,
\end{eqnarray}
and $\frac{\partial F_{\bar{*}}}{\partial s}|_{s=s_0}$ defined a.e. on $S^m_V$ is equivalent to a function in $L^p(S^m_V)$ for any $p\geq 1$ such that
\begin{eqnarray}\label{notaebok}
\int_{\{x\in M|a_1<F(s_0,x)<a_2\}} \frac{\partial F}{\partial s}|_{s=s_0} d\mu_g=\int_{\{y\in S^m_V|a_1<F_{\bar{*}}(s_0,y)<a_2\}} \frac{\partial F_{\bar{*}}}{\partial s}|_{s=s_0} d\mu_{\textsl{g}_V}
\end{eqnarray}
for any $[a_1,a_2]\subset \Bbb R$.

\end{Proposition}
\begin{proof}
Let $\{t_1,\cdots,t_l\}$ for $t_1<\cdots<t_l$ be the set of critical values of $F_{s_0}$.
First we prove (\ref{notaebok}) when $[a_1,a_2]$ is contained in some $(t_j,t_{j+1})$.
Applying (\ref{gguggu}) to (\ref{officetel-1}) and (\ref{officetel-2}),
we reach the desired conclusion that
$$\int_M \frac{\partial F}{\partial s}(s_0,x)\chi_{s_0}(x)\ d\mu_g=\int_{S^m_V} \frac{\partial F_{\bar{*}}}{\partial s}(s_0,y)\chi^*(y)\ d\mu_{\textsl{g}_V}.$$

Postponing the proof for the general case, we prove (\ref{Chokuk3}) by using the above-proven special case. Since $F_{\bar{*}}$ is smooth in a neighborhood of $\{s_0\}\times \{y\in S^m_V|r(y)=r_0\}$,
\begin{eqnarray*}
\left|\frac{\partial F_{\bar{*}}}{\partial s}(s_0,r_0)\right|&=&\lim_{\epsilon\rightarrow 0}\frac{|\int_{\{y\in S^m_V|\ |F_{\bar{*}}(s_0,y)-a_0|<\epsilon\}} \frac{\partial F_{\bar{*}}}{\partial s}|_{s=s_0}\ d\mu_{\textsl{g}_V}|}{\mu(\{y\in S^m_V|\ |F_{\bar{*}}(s_0,y)-a_0|<\epsilon\})}
\\ &=& \lim_{\epsilon\rightarrow 0}\frac{|\int_{\{x\in M|\ |F(s_0,x)-a_0|<\epsilon\}} \frac{\partial F}{\partial s}|_{s=s_0}\ d\mu_{g}|}{\mu(\{x\in M|\ |F(s_0,x)-a_0|<\epsilon\})}\\
&\leq& \max_{\{x\in M|F(s_0,x)=a_0\}}\left|\frac{\partial F}{\partial s}(s_0,x)\right|.
\end{eqnarray*}

Now by (\ref{Chokuk3}), $\frac{\partial F_{\bar{*}}}{\partial s}|_{s=s_0}$ is a bounded continuous function defined almost everywhere of $S^m_V$, i.e. away from $\bigcup_{i=1}^l(F_{\bar{*}s_0})^{-1}(t_i)$. Therefore it belongs to $L^p(S^m_V)$ for any $p\geq 1$ and (\ref{notaebok}) for general $[a_1,a_2]$ is obtained by using the above-proven special case of (\ref{notaebok}) combined with that $F_{s_0}^{-1}(t_i)$ and $(F_{\bar{*}s_0})^{-1}(t_i)$ for all $t_i$ are measure zero subsets of $M$ and $S^m_V$ respectively.
For example if $t_i=a_1<t_{i+1}<a_2<t_{i+2}$, then
\begin{eqnarray*}
\int_{\{x\in M|a_1<F(s_0,x)<a_2\}} \frac{\partial F}{\partial s}|_{s=s_0} d\mu_g&=&
\lim_{\epsilon\rightarrow 0}\int_{\{x\in M|F(s_0,x)\in I_\epsilon\}} \frac{\partial F}{\partial s}|_{s=s_0} d\mu_g\\
&=&
\lim_{\epsilon\rightarrow 0}\int_{\{y\in S^m_V|F_{\bar{*}}(s_0,y)\in I_\epsilon\}} \frac{\partial F_{\bar{*}}}{\partial s}|_{s=s_0} d\mu_{\textsl{g}_V}\\
&=&
\int_{\{y\in S^m_V|a_1<F_{\bar{*}}(s_0,y)<a_2\}} \frac{\partial F_{\bar{*}}}{\partial s}|_{s=s_0} d\mu_{\textsl{g}_V}
\end{eqnarray*}
where $I_\epsilon$ denotes $(a_1+\epsilon,t_{i+1}-\epsilon)\cup (t_{i+1}+\epsilon,a_2)$.

\end{proof}

\begin{Remark}\label{Tiatus}
By using (\ref{notaebok}) and (\ref{officetel-1}), we can generalize the equality (\ref{officetel-2}) to any $[a_1,a_2]\subset \Bbb R$.
Indeed for any $[a_1,a_2]$
\begin{eqnarray*}
\frac{\partial }{\partial s}|_{s=s_0}\int_{S^m_V} F_{\bar{*}}(s,y)\chi^*(y)\ d\mu_{\textsl{g}_V}(y)
&=& \frac{\partial }{\partial s}|_{s=s_0}\int_{M} F(s,x)\chi_{s}(x)\ d\mu_g(x)\\
&=& \int_M \frac{\partial F}{\partial s}(s_0,x)\chi_{s_0}(x)\ d\mu_g(x)\\
&=& \int_{S^m_V} \frac{\partial F_{\bar{*}}}{\partial s}(s_0,y)\chi^*(y)\ d\mu_{\textsl{g}_V}(y).
\end{eqnarray*}
\end{Remark}

Now we can state the conclusion of this subsection. As stated in Introduction, we shall split the exterior derivative $d$ on $N\times M$ and $d$ on $N\times S^m_V$ as $d^N+d^M$ and $d^N+d^S$ respectively. The inequalities in Theorem \ref{forgive-me} immediately follow from the first two inequalities in :
\begin{Proposition}\label{yun}
For any $s\in N_0$, $[a_1,a_2]\subset \Bbb R$, and $p\geq 1$
$$\int_{\{y\in S^m_V|a_1<F_{\bar{*}}(s,y)<a_2\}} \left|d^N F_{\bar{*}}\right|^p d\mu_{\textsl{g}_V} \leq
\int_{\{x\in M|a_1<F(s,x)<a_2\}} \left|d^N F\right|^p d\mu_g$$
$$\int_{\{y\in S^m_V|a_1<F_{\bar{*}}(s,y)<a_2\}} \left|d^S F_{\bar{*}}\right|^p d\mu_{\textsl{g}_V} \leq (C_1)^p
\int_{\{x\in M|a_1<F(s,x)<a_2\}} \left|d^M F\right|^p d\mu_g$$
$$\sup_{N_0\times S^m_V-\mathcal{C}^*}|dF_{\bar{*}}|\leq \max(1,C_1)\max_{N\times M}|dF|$$
where $C_1$ is the constant in Proposition \ref{sejong-brt}, and if $Ric_g \geq k > 0$, then $\max(1,C_1)$ is given by $C_1=\left(\frac{V_m}{V}\right)^{\frac{1}{m}}\sqrt{\frac{m-1}{k}}$.

The extended function $F_{\bar{*}}:N\times S^m_V\rightarrow \Bbb R$ is weakly differentiable belonging to $L_1^p(N\times S^m_V)$ for any $p\geq 1$.
\end{Proposition}
\begin{proof}
For the 1st inequality
we need a manipulation on (\ref{notaebok}). As before, the fiberwise gradients on each $\{s\}\times M$ and $\{s\}\times S^m_V$ are denoted by $\nabla^M$ and $\nabla^S$ respectively, so $$|\nabla^MF|=|d^MF| \ \ \ \textrm{and}\ \ \ |\nabla^SF_{\bar{*}}|=|d^SF_{\bar{*}}|.$$
By Proposition \ref{blessed}, $\frac{\partial F_{\bar{*}}}{\partial s}|_{\{s\}\times S^m_V}$ for $s\in N_0$ belongs to $L^1(S^m_V)$ and is smooth around $\Sigma_{s,a}^*={\{y\in S^m_V|F_{\bar{*}}(s,y)=a\}}$ for a regular value $a$ of $F_s:M\rightarrow \Bbb R$. Thus one can apply the coarea formula and (\ref{notaebok}) to get
\begin{eqnarray}\label{sonshine}
\int_{\Sigma_{s,a}} \frac{\partial F}{\partial s}  {| \nabla^M F_s | }^{-1} d\sigma_a &=&
\frac{d}{dt}|_{t=a}\int_{-\infty}^t\int_{\Sigma_{s,t}} \frac{\partial F}{\partial s}  {| \nabla^M F_s | }^{-1} d\sigma_t dt \nonumber\\ &=&
\frac{d}{dt}|_{t=a}\int_{\{x\in M|F(s,x)<t\}} \frac{\partial F}{\partial s}\ d\mu_g\nonumber\\ &=&
\frac{d}{dt}|_{t=a}\int_{\{y\in S^m_V|F_{\bar{*}}(s,y)<t\}} \frac{\partial F_{\bar{*}}}{\partial s}\ d\mu_{\textsl{g}_V} \nonumber\\ &=&
\frac{d}{dt}|_{t=a}\int_{-\infty}^t\int_{\Sigma_{s,t}^*} \frac{\partial F_{\bar{*}}}{\partial s}  {| \nabla^S F_{\bar{*}s}| }^{-1} d\sigma_t dt \nonumber\\ &=& \int_{\Sigma_{s,a}^*} \frac{\partial F_{\bar{*}}}{\partial s}  {| \nabla^S F_{\bar{*}s}| }^{-1} d\sigma_a.
\end{eqnarray}
Let $\{t_1,\cdots,t_l\}$ be the set of critical values of $F_s$ lying in $(a_1,a_2)$ for $s\in N_0$, and set $t_0:=a_1, t_{l+1}:=a_2$.
We combine the coarea formula, (\ref{sonshine}), and (\ref{WTH}) to obtain
\begin{eqnarray*}
& & \int_{\{x\in M|a_1<F(s,x)<a_2\}} \left|\frac{\partial F}{\partial s}\right|^pd\mu_g\\
 &=&
 \sum_{i=0}^{l}\int_{t_{i}}^{t_{i+1}} \left(\int_{\Sigma_{s,t}}\left|\frac{\partial F}{\partial s}\right|^p {| \nabla^M F_s| }^{-1} d\sigma_t\right)dt\\ &\geq&  \sum_{i=0}^{l}\int_{t_{i}}^{t_{i+1}} \left|\int_{\Sigma_{s,t}}\frac{\partial F}{\partial s} {| \nabla^M F_s | }^{-1} d\sigma_t\right|^p\left(\int_{\Sigma_{s,t}}{| \nabla^M F_s| }^{-1} d\sigma_t\right)^{-\frac{p}{q}}dt\\
&=&  \sum_{i=0}^{l}\int_{t_{i}}^{t_{i+1}} \left|\int_{\Sigma_{s,t}^*}\frac{\partial F_{\bar{*}}}{\partial s} {| \nabla^S F_{\bar{*}s} | }^{-1} d\sigma_t\right|^p\left(\int_{\Sigma_{s,t}^*}{| \nabla^S F_{\bar{*}s} | }^{-1} d\sigma_t\right)^{-\frac{p}{q}}dt\\
&=&  \sum_{i=0}^{l}\int_{t_{i}}^{t_{i+1}} \left(\int_{\Sigma_{s,t}^*}\left|\frac{\partial F_{\bar{*}}}{\partial s}\right|^p {| \nabla^S F_{\bar{*}s} | }^{-1} d\sigma_t\right)dt\\
&=& \int_{\{y\in S^m_V|a_1<F_{\bar{*}}(s,y)<a_2\}} \left|\frac{\partial F_{\bar{*}}}{\partial s}\right|^p d\mu_{\textsl{g}_V},
\end{eqnarray*}
where the inequality of the 3rd line is due to H\"older inequality with $\frac{1}{p}+\frac{1}{q}=1$ applied to the integration on $\Sigma_{s,t}$ for each $t\notin \{t_1,\cdots,t_l\}$ with respect to a volume element ${| \nabla^M F_s | }^{-1} d\sigma_t$, and the equality of the 5th line is due to the radial symmetry of $F_{\bar{*}}$ along $S^m_V$. The 1st inequality follows from this integral inequality.

The 2nd and 3rd inequalities are just the restatements of Proposition \ref{sejong-brt} and (\ref{Chokuk3}) of Proposition \ref{blessed}.
In case $Ric_g\geq k$, from (\ref{itismyfault-1}) we know $C_1=(\frac{V_m}{V})^{\frac{1}{m}}\sqrt{\frac{m-1}{k}}$.
Since $(\frac{V_m}{V})^{\frac{1}{m}}\sqrt{\frac{m-1}{k}}$ is invariant under a homothety $g\mapsto a^2g$ for any $a>0$, by considering the case when $Ric_g\geq k=m-1$ one can see that it becomes $(\frac{V_m}{V})^{\frac{1}{m}}\geq 1$. Thus $\max(1,C_1)$ is $C_1=(\frac{V_m}{V})^{\frac{1}{m}}\sqrt{\frac{m-1}{k}}$.

From Proposition \ref{hanguk}, we know that $F_{\bar{*}}$ is extended to a Lipschitz continuous function on $N\times S^m_V$.
To show its weak differentiability, we use Proposition \ref{Sewoong}.
The restriction of $F_{\bar{*}}$ to every coordinate line of a local coordinate on $N\times S^m_V$ is Lipschitz continuous and hence absolutely continuous.
Since the 1st order strong partial derivatives of $F_{\bar{*}}$ defined a.e. in $N\times S^m_V$ are locally integrable as shown in the above, we can conclude that $F_{\bar{*}}$ is weakly differentiable in $N\times S^m_V$ and
belongs to $L_1^p(N\times S^m_V)$ for any $p\geq 1$.

\end{proof}

\subsection{The continuity of spherically-rearranging map}

Meanwhile we have the generalization of Proposition \ref{Firstlady-1} to the fiberwise symmetrization :
\begin{Theorem}\label{Firstlady}
Suppose $(M^m,g)$ and $(N^n,h)$ are smooth closed Riemannian manifolds such that $M$ has volume $V$ and $Ric_g>0$.
Let $$\mathcal{F}:=\{F:N\times M\stackrel{C^\infty}{\rightarrow} \Bbb R|F \textrm{ is a generic-fiberwise Morse function.}\}$$ and $$\tilde{\mathcal{F}}:=\{F_{\bar{*}}:N\times S^m_V\rightarrow \Bbb R|F\in \mathcal{F}\}$$ be endowed with $C^0$-topology.
Then the spherically-rearranging map $\Phi:\mathcal{F}\rightarrow \tilde{\mathcal{F}}$  given by $\Phi(F)=F_{\bar{*}}$ is Lipschitz continuous with Lipschitz constant 1,
i.e. for any $F, G\in \mathcal{F}$   $$||G_{\bar{*}}-F_{\bar{*}}||_\infty\leq ||G-F||_\infty.$$
\end{Theorem}
\begin{proof}
The proof is similar to that of Proposition \ref{Firstlady-1}.
To prove by contradiction, suppose not.
Take $F,G\in \mathcal{F}$ such that $$||G_{\bar{*}}-F_{\bar{*}}||_\infty> ||G-F||_\infty.$$

Since $N_0(F)\cap N_0(G)$ is dense in $N$ by Baire's category theorem,
there exist $s_0\in N_0(F)\cap N_0(G)$ and $r_0\in (0,R_V)$ such that $$|G_{\bar{*}}(s_0,r_0)-F_{\bar{*}}(s_0,r_0)|>||G-F||_\infty.$$
($G_{\bar{*}}$ and $F_{\bar{*}}$ are understood as functions on $N\times [0,R_V]$.)

Without loss of generality we may assume that $$F_{\bar{*}}(s_0,r_0)> G_{\bar{*}}(s_0,r_0)+||G-F||_\infty$$ holds, and we set $a:=G_{\bar{*}}(s_0,r_0)$.
Then
\begin{eqnarray*}
\mu(\{y\in S^m_V|r(y)<r_0\}) &=& \mu(\{y\in S^m_V|G_{\bar{*}}(s_0,r(y))< a  \}) \\ &=& \mu(\{x\in M|G(s_0,x)<  a  \})\\ &\leq& \mu(\{x\in M|F(s_0,x)<  a+||G-F||_\infty  \}) \\
&=& \mu(\{y\in S^m_V|F_{\bar{*}}(s_0,r(y))<  a+||G-F||_\infty  \})\\ &<& \mu(\{y\in S^m_V|F_{\bar{*}}(s_0,r(y))<  F_{\bar{*}}(s_0,r_0) \})\\ &=& \mu(\{y\in S^m_V|r(y)<r_0\})
\end{eqnarray*}
which yields a contradiction.
\end{proof}

\begin{Corollary}
Under the same hypothesis as the above theorem, $\Phi$ can be uniquely extended to the Lipschitz continuous map $$\Phi:C^0(N\times M)\rightarrow C^0(N\times S^m_V)$$
with Lipschitz constant 1.
\end{Corollary}
\begin{proof}
We have shown that $\mathcal{F}$ is dense in $C^\infty(N\times M)$ in $C^2$-topology and hence $C^0$-topology as well. Since $C^\infty(N\times M)$ is dense in  $C^0(N\times M)$ w.r.t. $C^0$-topology, $\mathcal{F}$ is dense in $C^0(N\times M)$. Now the Lipschitz continuity of $\Phi$ on $\mathcal{F}$ ensures such an extension.
\end{proof}

\section{Euclidean and hyperbolic rearrangements}\label{SWYang}

In the same way as the spherical rearrangement, one can take a rearrangement on the other isoperimetric model spaces, i.e. $(\Bbb R^m,g_{_{\Bbb E}})$ or Poincar\'{e} $m$-disk $(\Bbb H^m,g_{_{\Bbb H}})$, and obtain all the results corresponding to those obtained so far.

\subsection{Rearrangement in Euclidean space}
Let $(M^m,g)$ be a smooth closed Riemannian $m$-manifold with volume $V$.
The Euclidean rearrangement of a smooth Morse function $f:M\rightarrow \Bbb R$  is the radially-symmetric continuous function $f_\star: \overline{D^m_V}\rightarrow \Bbb R$ such that $\{y\in \overline{D^m_V}| f_\star(y)<t \}$ for any $t\in (\min(f), \max(f)]$ is an open $m$-disc centered at the origin with the same volume as $\{x\in M| f(x)<t \}$.

More generally its parametrized version is defined just as in the spherical rearrangement. When $(N,h)$ is a smooth closed Riemannian $n$-manifold, and $F:N\times M\rightarrow \Bbb R$ is a generic-fiberwise Morse function such that $F_s$ for $s\in N_0=N_0(F)\subseteq N$ is a Morse function,
the fiberwise Euclidean rearrangement of $F$ is $F_{\bar{\star}}:N_0\times \overline{D^m_V}\rightarrow \Bbb R$ satisfying that
$F_{\bar{\star}s}:=(F_{\bar{\star}})_s$ for any $s\in N_0$ is equal to $(F_s)_{\star}$.
Hence $$\int_{\{x\in M| a<F_s(x)<b \}}F_s^k\ d\mu_g=\int_{\{y\in D^m_{V}|a<F_{\bar{\star}s}(y)<b\}} (F_{\bar{\star}s})^k\ d\mu_{g_{_{\Bbb E}}}$$ for any $k>0$, $s\in N_0$, and $[a,b]\subset \Bbb R$, where $(\cdot)^k$ means $|\cdot|^k$ if $k$ is not an integer.

Since $F_{\bar{\star}}$ is radially symmetric, we shall also regard it as a function on $N_0\times [0,R_V]$ for $R_V:=\frac{1}{2}\textrm{diam}(D^m_V)$. Exactly the same way as $F_{\bar{*}}$ verifies that $F_{\bar{\star}}:N_0\times \overline{D^m_V}\rightarrow \Bbb R$ is continuous everywhere and smooth away from $$\mathcal{C}^\star:=\{(s,y)\in N_0\times \overline{D^m_V}|(s,F_{\bar{\star}}(s,y))\in F(\mathcal{C})\}$$ where $\mathcal{C}$ is the set in (\ref{LHW}).
If $t_0\in \Bbb R$ is a regular value of $F_{s}$ for $s\in N_0$, then
$$
\int_{F^{-1}_{s}(t_0)} {| \nabla^M F_{s} |}^{-1} d\sigma_{t_0} =\int_{(F_{\bar{\star}s})^{-1}(t_0)} |\nabla^S F_{\bar{\star}s} |^{-1} d\sigma_{t_0}
$$
where $\nabla^M$ and $\nabla^S$ mean the fiberwise gradients on each $\{s\}\times M$ and $\{s\}\times \overline{D^m_V}$ respectively.

All the followings are proved verbatim from those of $F_{\bar{*}}$. Let $s$ be a local coordinate function of $N_0$ and $a_1,a_2$ be regular values of $F_{s_0}$ for $s_0\in N_0$. Defining $\chi_{s}$ and $\chi^\star$ according to $a_1$ and $a_2$ just as in the case of $F_{\bar{*}}$,
we have
\begin{eqnarray*}
\frac{\partial }{\partial s}|_{s=s_0}\int_{M} F_s(x)\chi_{s}(x)\ d\mu_g(x) &=& \int_M \frac{\partial F}{\partial s}(s_0,x)\chi_{s_0}(x)\ d\mu_g(x),
\end{eqnarray*}
and if \footnote{In fact this condition is unnecessary as we have shown in Remark \ref{Tiatus}.} an interval $[a_1,a_2]$ does not contain a critical value of $F_{s_0}$,
\begin{eqnarray*}
\frac{\partial }{\partial s}|_{s=s_0}\int_{D^m_V} F_{\bar{\star}s}(y)\chi^\star(y)\ d\mu_{g_{_{\Bbb E}}}(y)
&=&\int_{D^m_V} \frac{\partial F_{\bar{\star}}}{\partial s}(s_0,y)\chi^\star(y)\ d\mu_{g_{_{\Bbb E}}}(y).
\end{eqnarray*}
Thus for any $[a_1,a_2]\subset \Bbb R$ and any $p\geq 1$, $$\int_{\{x\in M|a_1<F_{s}(x)<a_2\}} \frac{\partial F}{\partial s}\ d\mu_g=\int_{\{y\in D^m_V|a_1<F_{\bar{\star}s}(y)<a_2\}} \frac{\partial F_{\bar{\star}}}{\partial s}\ d\mu_{g_{_{\Bbb E}}}$$
\begin{eqnarray*}
\int_{\{y\in D^m_V|a_1<F_{\bar{\star}s}(y)<a_2\}} \left|\frac{\partial F_{\bar{\star}}}{\partial s}\right|^p d\mu_{g_{_{\Bbb E}}}
\leq
\int_{\{x\in M|a_1<F_s(x)<a_2\}} \left|\frac{\partial F}{\partial s}\right|^p d\mu_g,
\end{eqnarray*}
and for any regular value $a_0=F_{\bar{\star}}(s,r_0)$ of $F_{s}$
\begin{eqnarray*}
\left|\frac{\partial F_{\bar{\star}}}{\partial s}(s,r_0)\right| &\leq& \max_{\{x\in M|F_{s}(x)=a_0\}}\left|\frac{\partial F}{\partial s}(s,x)\right|
\end{eqnarray*}
$$|\nabla^S F_{\bar{\star}s}(r_0)|\leq C_1'\max_{\{x\in M|F_s(x)=a_0\}}|\nabla^M F_s(x)|$$ for a constant $C_1'$ depending only on $(M,g)$.
Finally for $$\mathcal{F}:=\{f:M\stackrel{C^\infty}{\rightarrow} \Bbb R|f \textrm{ is a Morse function.}\}$$ and
$$\hat{\mathcal{F}}:=\{f_{\star}: \overline{D^m_V}\rightarrow \Bbb R|f\in \mathcal{F}\}$$ both endowed with $C^0$-topology,
the map $\Phi:\mathcal{F}\rightarrow \hat{\mathcal{F}}$ given by $\Phi(f)=f_\star$ is Lipschitz continuous with Lipschitz constant 1.

To exactly estimate the $L_1^2$-norm of $F_{\bar{\star}}$, we need to determine the constant $C_1'$ in the above, and hence we need a curvature condition on $M$.
\begin{Proposition}\label{Haemu}
Let $f:M\rightarrow \Bbb R$ be a smooth Morse function. Suppose that an open subset $B\subset M$ is isometric to a bounded
domain of $(\Bbb R^m, g_{_{\Bbb E}})$ and $f^{-1}(-\infty,b]$ for $b\in \Bbb R$ is contained in $B$.
Then for any $a<b$ and any $p\geq 1$
\begin{eqnarray*}
\int_{f_\star^{-1}(a,b]} {| \nabla f_\star| }^p\ d\mu_{g_{_{\Bbb E}}}
\leq
\int_{f^{-1}(a,b]} {| \nabla f | }^p\ d\mu_{g_{_{\Bbb E}}},
\end{eqnarray*}
and for any regular value $a_0=f_\star(r_0)\in [a,b]$ of $f$
\begin{eqnarray*}
|\nabla f_\star(r_0)|\leq \max_{\{x\in M|f(x)=a_0\}}|\nabla f(x)|.
\end{eqnarray*}
\end{Proposition}
\begin{proof}
The basic idea of its proof is the same as Proposition \ref{pet} and also well-known as a standard method to prove the Faber-Krahn inequality.(\cite{SY})

Let $t\in [a,b]$ be a regular value of $f$.
Using the facts that $\{x\in B| f(x)< t \}$ is an open submanifold of $M$ with smooth boundary $f^{-1} (t)$, and an $m$-dimensional disc has the smallest boundary area  among all hypersurfaces in $\Bbb R^m$ enclosing the same volume, we have $$\mu (f^{-1} (t)) \geq \mu ( f_\star^{-1} (t))$$ instead of (\ref{uncultured}).
The rest of proof proceeds in the same way as in Proposition \ref{pet}.
\end{proof}

By combining all the above the following conclusion for $F_{\bar{\star}}$ is derived in the same way as $F_{\bar{*}}$.
\begin{Theorem}\label{Young&Hyuk}
Let $(M^m,g)$ with volume $V$ and $(N^n,h)$ be smooth closed Riemannian manifolds and $F:N\times M \rightarrow \Bbb R$ be a smooth generic-fiberwise Morse function.
Then $F_{\bar{\star}}$ extends to a unique radially-symmetric Lipschitz continuous $L_1^p$ function in $N\times \overline{D^m_V}$ for any $p\geq 1$.

If an open subset $\mathfrak{B}\subset (M,g)$ is isometric to a bounded domain of $(\Bbb R^m, g_{_{\Bbb E}})$,
and $F^{-1}(-\infty,b]$ for $b\in \Bbb R$ is contained in $N\times \mathfrak{B}$, then for any $a<b$
\begin{eqnarray*}
\int_{F^{-1}_{\bar{\star}}(a,b]} |dF_{\bar{\star}}|^p\ d\mu_{h+g_{_{\Bbb E}}} \leq \int_{F^{-1}(a,b]} |dF|^p\ d\mu_{h+g}.
\end{eqnarray*}
\end{Theorem}

\subsection{Rearrangement in hyperbolic space}
Now let $\mathfrak{D}^m_V$ denote the open $m$-disk in $(\Bbb H^m,g_{_{\Bbb H}})$ with center at the origin and hyperbolic volume $V$.
The hyperbolic rearrangement of a smooth Morse function $f:M\rightarrow \Bbb R$ defined on $(M^m,g)$ with volume $V$ is the radially-symmetric continuous function $f_\bullet: \overline{\mathfrak{D}^m_V}\rightarrow \Bbb R$ such that $\{y\in \overline{\mathfrak{D}^m_V}| f_\bullet(y)<t \}$ for any $t\in (\min(f), \max(f)]$ is the geodesic ball centered at the origin with hyperbolic volume equal to $\mu(\{x\in M| f(x)<t \})$.

One can check that the results corresponding to those of the previous subsection hold in the same way by replacing $(\Bbb R^m, g_{_{\Bbb E}})$ with  $(\Bbb H^m,g_{_{\Bbb H}})$ and using the fact that every geodesic ball in $\Bbb H^m$ is an isoperimetric region.

\section{Proof of Theorem \ref{Yoon-1}}

First recall that on smooth closed Riemannian manifold $(X,{\textrm{g}})$, the first eigenvalue $\lambda_1$ of its Laplacian $\Delta=d^*d+dd^*$ acting on $C^\infty(X)$ is represented as
\begin{eqnarray*}
\lambda_1=\min\{R_{\textrm{g}}(f)|f\in L_1^2(X)\backslash\{0\}, \int_Xf\ d\mu_{\textrm{g}}=0\}
\end{eqnarray*}
where $R_{\textrm{g}}(f)$ is the Rayleigh quotient of $f$ defined as $$\frac{\int_X |df|_{\textrm{g}}^2\ d\mu_{\textrm{g}}}{\int_X f^2\ d\mu_{\textrm{g}}}.$$

Let us fix some notations. As usual, $V$ denotes the volume of $(M,g)$ and set $$\textbf{h}_V:=\left(\frac{V}{V_m}\right)^{\frac{2}{m}}h.$$
By $|*|_\zeta$ we shall mean the norm of $*$ with respect to a metric $\zeta$. Obviously
\begin{eqnarray}\label{JHK0}
\left( \frac{V_m}{V} \right)^{\frac{n}{m}}\int_{N}f\ d\mu_{\textbf{h}_V}=\int_Nf\ d\mu_h, \ \ \ \ \ \  \ \ \ \left(\frac{V}{V_m}\right)^{\frac{2}{m}}| d^{N}f|_{\textbf{h}_V}^2=| d^{N}f|_{h}^2
\end{eqnarray}
for any smooth $f:N\rightarrow \Bbb R$.

Let $\epsilon\ll 1$ be a positive number and $\varphi: N\times M \rightarrow \Bbb R$ be the first eigenfunction of $(N\times M,h+\rho^2g)$ such that $\int_{N\times M}\varphi^2\ d\mu_{h+\rho^2g}=1$. Take a smooth generic-fiberwise Morse function
$\tilde{\varphi} :N\times M \rightarrow \Bbb R$ which is $C^2$-close enough to $\varphi$ so that
$$|R_{h+\rho^2g} (\varphi)-R_{h+\rho^2g}(\tilde{\varphi})|< \epsilon
,\ \ \ \textrm{and}\ \ \ \int_{N\times M}\tilde{\varphi}\ d\mu_{h+\rho^2g}=0$$ where the 2nd condition can be achieved by using $\int_{N\times M}\varphi\ d\mu_{h+\rho^2g}=0$ and adding a small constant as necessary.
By Proposition \ref{yun} and (\ref{JHK0}),
\begin{eqnarray}\label{JHK1}
\int_M| d^M \tilde{\varphi}|^2_g\ d\mu_g\geq {\left( \frac{V}{V_m} \right) }^{\frac{2}{m}}\int_{S^m_{V}}| d^S \tilde{\varphi}_{\bar{*}}|_{\textsl{g}_V}^2d\mu_{\textsl{g}_V}
\end{eqnarray}
\begin{eqnarray}\label{JHK2}
\int_M| d^N \tilde{\varphi}|^2_h\ d\mu_g \geq  {\left( \frac{V}{V_m} \right) }^{\frac{2}{m}}\int_{S^m_{V}}| d^{N} \tilde{\varphi}_{\bar{*}}|^2_{\textbf{h}_V}d\mu_{\textsl{g}_V}
\end{eqnarray}
at any point of $N_0=N_0(\tilde{\varphi})$.

We need to check that $\tilde{\varphi}_{\bar{*}}$ is not zero and still $L^2$-orthogonal to 1.
Indeed
\begin{eqnarray*}
\int_{N\times S^m}\tilde{\varphi}_{\bar{*}}\ d\mu_{\textbf{h}_V+\rho^2\textsl{g}_V} &=& \left( \frac{V}{V_m} \right)^{\frac{n}{m}}\int_N\left(\int_{S^m}\tilde{\varphi}_{\bar{*}}\ d\mu_{\textsl{g}_V}\right)\rho^md\mu_h\\
&=& \left( \frac{V}{V_m} \right)^{\frac{n}{m}}\int_N\left(\int_{M}\tilde{\varphi}\ d\mu_{g}\right)\rho^md\mu_h \\
&=& \left( \frac{V}{V_m} \right)^{\frac{n}{m}}\int_{N\times M}\tilde{\varphi}\ d\mu_{h+\rho^2g}\\
&=& 0
\end{eqnarray*}
and
\begin{eqnarray*}
\int_{N\times S^m}(\tilde{\varphi}_{\bar{*}})^2\ d\mu_{\textbf{h}_V+\rho^2\textsl{g}_V}
&=& \left( \frac{V}{V_m} \right)^{\frac{n}{m}}\int_N\left(\int_{S^m}(\tilde{\varphi}_{\bar{*}})^2\ d\mu_{\textsl{g}_V}\right)\rho^md\mu_h\\
&=& \left( \frac{V}{V_m} \right)^{\frac{n}{m}}\int_N\left(\int_{M}\tilde{\varphi}^2\ d\mu_{g}\right)\rho^md\mu_h \\
&=& \left( \frac{V}{V_m} \right)^{\frac{n}{m}}\int_{N\times M}\tilde{\varphi}^2\ d\mu_{h+\rho^2g}\\ &\neq& 0.
\end{eqnarray*}

Using (\ref{gorba}, \ref{JHK0}, \ref{JHK1}, \ref{JHK2}), we get
\begin{eqnarray*}
R_{h+\rho^2g} (\tilde{\varphi}) &=&\frac{\int_{N\times M} | d\tilde{\varphi}|_{h+\rho^2g}^2d\mu_{h+\rho^2g}}{\int_{N\times M}\tilde{\varphi}^2\ d\mu_{h+\rho^2g}}\\ &=& \frac{\int_{N}\rho^m\int_M \left(\rho^{-2}| d^M \tilde{\varphi}|_g^2+{| d^N \tilde{\varphi}|_h }^2\right) d\mu_{g}d\mu_{h}}{\int_{N}\rho^m\int_M\tilde{\varphi}^2\ d\mu_{g}d\mu_{h}}\\  &\geq& \frac{\left( \frac{V_m}{V} \right) ^{\frac{n}{m}}\int_{N}\rho^m\int_{S^m_{V}}(\frac{V}{V_m})^{\frac{2}{m}}\left(\rho^{-2}| d^S \tilde{\varphi}_{\bar{*}} |_{\textsl{g}_V}^2 +| d^{N} \tilde{\varphi}_{\bar{*}}|_{\textbf{h}_V}^2\right)d\mu_{\textsl{g}_V}d\mu_{\textbf{h}_V}}{ {\left( \frac{V_m}{V} \right) ^{\frac{n}{m}}\int_{N}\rho^m\int_{S^m_{V}}
\tilde{\varphi}_{\bar{*}}^2}\ d\mu_{\textsl{g}_V}d\mu_{\textbf{h}_V}}\\
&=& {\left( \frac{V}{V_m} \right) }^{\frac{2}{m}}R_{\textbf{h}_V+\rho^2\textsl{g}_V} (\tilde{\varphi}_{\bar{*}} )\\  &\geq& \left(\frac{V}{V_m}\right)^{\frac{2}{m}}\lambda_1(N\times S^m,\textbf{h}_V+\rho^2\textsl{g}_V).
\end{eqnarray*}
Therefore
\begin{eqnarray*}
\lambda_1(N\times M,h+\rho^2g) &=& R_{h+\rho^2g} (\varphi) \\ &>& R_{h+\rho^2g} (\tilde{\varphi})-\epsilon\\
 &\geq& \left(\frac{V}{V_m}\right)^{\frac{2}{m}}\lambda_1(N\times S^m,\textbf{h}_V+\rho^2\textsl{g}_V)-\epsilon\\
&=& \lambda_1(N\times S^m,h+\rho^2g_{_{\Bbb S}})-\epsilon
\end{eqnarray*}
where the last equality is simply the application of a general fact $$\lambda_1(X,a^2{\textrm{g}})=\frac{1}{a^2}\lambda_1(X,{\textrm{g}}),$$
and the desired inequality is obtained by taking $\epsilon>0$ arbitrarily small.


\bigskip

\noindent{\bf Acknowledgement} The author would like to express his deepest gratitude to God, parents, and all the teachers who taught him mathematics. Much of this work was done while on sabbatical, so the author is grateful to his home country for allowing him such a valuable opportunity.





\vspace{0.5cm}

\end{document}